\documentclass[reqno, 12pt]{amsart}  
\theoremstyle{plain}
\newtheorem{theorem}{Theorem}[section]

\newtheorem{remark}[theorem]{Remark}
\newtheorem{proposition}[theorem]{Proposition}

\newtheorem{corollary}[theorem]{Corollary}

\numberwithin{equation}{section}
\allowdisplaybreaks[1]

\theoremstyle{definition}

\theoremstyle{remark}

\newcommand{\bS}{{\mathbf S}}
\newcommand{\be}{{\mathbf e}}

\newcommand{\bQ}{{\mathbf Q}}

\newcommand{\bU}{{\mathbf U}}

\newcommand{\cD}{{\mathcal D}}
\newcommand{\cF}{{\mathcal F}}
\newcommand{\cH}{{\mathcal H}}
\newcommand{\cL}{{\mathcal L}}
\newcommand{\cM}{{\mathcal M}}
\newcommand{\cR}{{\mathcal R}}
\newcommand{\cK}{{\mathcal K}}

\newcommand{\cS}{{\mathcal S}}
\newcommand{\cE}{{\mathcal E}}

\newcommand{\cU}{{\mathcal U}}
\newcommand{\cX}{{\mathcal X}}
\newcommand{\cY}{{\mathcal Y}}
\newcommand{\cO}{{\mathcal O}}
\newcommand{\cG}{{\mathcal G}}
\newcommand{\D}{{\mathbb D}}

\begin{document}

\title[de Branges-Rovnyak spaces]{de Branges-Rovnyak spaces: basics and theory}
\author[J. A. Ball]{Joseph A. Ball}
\address{Department of Mathematics,
Virginia Tech,
Blacksburg, VA 24061-0123, USA}
\email{joball@math.vt.edu}
\author[V. Bolotnikov]{Vladimir Bolotnikov}
\address{Department of Mathematics,
The College of William and Mary,
Williamsburg VA 23187-8795, USA}
\email{vladi@math.wm.edu}

\begin{abstract}
For $S$ a contractive analytic operator-valued function on the 
unit disk ${\mathbb D}$, de Branges and Rovnyak associate a  
Hilbert space of analytic functions $\cH(S)$ and related 
extension space $\cD(S)$ consisting of pairs of analytic functions 
on the unit disk ${\mathbb D}$. This survey describes three 
equivalent formulations (the original geometric de Branges-Rovnyak 
definition, the Toeplitz operator characterization, and the 
characterization as a reproducing kernel Hilbert space) of the de 
Branges-Rovnyak space $\cH(S)$, as well as its role as the underlying Hilbert 
space for the modeling of completely non-isometric Hilbert-space contraction 
operators.  Also examined is the extension of these ideas to handle 
the modeling of the more general class of completely nonunitary 
contraction operators, where the more general two-component de 
Branges-Rovnyak model space $\cD(S)$ and associated overlapping 
spaces play key roles. Connections with other function theory 
problems and applications are also discussed.  More recent 
applications to a variety of subsequent applications are given in a 
companion survey article.
\end{abstract}

\subjclass{47A57}
\keywords{de Branges-Rovnyak spaces, lifted norm, Brangesian 
complementary space, reproducing kernel Hilbert space}

\maketitle
\tableofcontents

\section{Introduction}
 In the late 1960s and early 1970s, 
Louis de Branges and James Rovnyak introduced and studied spaces of 
vector-valued holomorphic functions on the open unit disk $\D$ associated with 
what is now called a Schur-class function $S \in \cS(\cU, \cY)$ 
(i.e., a holomorphic function $S$ on the unit disk with values equal 
to contraction operators between Hilbert coefficient spaces $\cU$ and 
$\cY$---although in the original work of de Branges and Rovnyak the 
choice $\cU = \cY$ was usually taken).  These spaces were related to but distinct from 
the Hilbert spaces of entire functions explored in earlier work of de 
Branges (see in particular the book \cite{dB68}); these latter spaces 
in turn have been revived recently, especially in the work of H.~Dym 
and associates (see \cite{DMcK76, ArovDym1, ArovDym2}) as well as 
others and have deep 
connections with the work of M.G.~Kre\u{\i}n and assorted applied 
problems (e.g., continuous analogs of orthogonal polynomials and 
associated moment problems, inverse string problems).  These spaces 
also serve as model spaces for unbounded densely defined symmetric 
operators with equal deficiency indices. As other authors will be 
discussing these spaces in other chapters of this series, our focus here 
will be on the de Branges-Rovnyak 
spaces on the unit disk. Motivation for the study of these spaces 
seems to be from at least two sources:  
\begin{enumerate}
\item quantum scattering theory
(see \cite{dBR1} as well as the papers \cite{dB77, dBS}), and 
\item operator model 
theory for Hilbert space contraction operators and the invariant 
subspace problem (see 
\cite[Appendix]{dBR1} and \cite{dBR2}).
\end{enumerate}
The connection with quantum scattering had to do with using 
the machinery of Hilbert spaces of analytic functions (in particular, 
an object called overlapping spaces) to set up a formalism for the 
study of the perturbation theory for self-adjoint operators (or equivalently after 
Cayley transformation, to the perturbation theory of unitary 
operators), an important topic in the wave-operator approach to 
scattering theory.  This article does not go into this topic, but rather 
focuses on the second application, namely, to operator model theory.

There are now at least three 
distinct ways of introducing the de Branges-Rovnyak spaces: 
\begin{enumerate}
\item  the 
original definition of de Branges and Rovnyak (as the complementary 
space of $S \cdot H^{2}$), 
\item  as the range of the Toeplitz defect 
operator with lifted norm, or 
\item  as the reproducing kernel Hilbert 
space with reproducing kernel given by the de Branges-Rovnyak 
positive kernel.  
\end{enumerate}
In the next three sections, each of these will be discussed in turn. 

\section{The original de Branges-Rovnyak formulation}

In what follows, the symbol
$\cL(\cU,\cY)$ stands for the space of bounded linear operators mapping a Hilbert
space $\cU$ into a Hilbert space $\cY$, abbreviated to $\cL(\cY)$ in case $\cU=\cY$.
The standard  Hardy space of $\cY$-valued functions on
the open unit disk $\D$ with square-summable sequences of Taylor 
coefficient  is denoted by $H^2(\cY)$ and the notation $\cS(\cU,\cY)$ 
is used for the Schur class of functions analytic on $\D$
whose values are contractive operators in $\cL(\cU,\cY)$.

\smallskip 

Let $S\in\cS(\cU,\cY)$ be a Schur-class function. L. de Branges and J. Rovnyak 
define the space $\cH(S)$ according to the prescription
\begin{equation}
\cH(S) =\{ f \in H^{2}(\cU) \colon \| f\|^{2}_{\cH(S)}: = \sup_{g \in
H^{2}(\cU)}\{
\|f + Sg \|^{2}_{H^{2}(\cY)}  - \| g \|^{2}_{H^{2}(\cU)}\} < \infty\}.\label{def1HS}
\end{equation}
At first glance the definition looks rather impenetrable, except for 
one easy special case, namely, the case where $S$ is inner.  In this 
case, it is relatively straightforward to see that $\cH(S)$ is 
isometrically equal to $H^{2}(\cY) \ominus S \cdot H^{2}(\cU)$. 
Nevertheless, it is possible to show directly from the definition 
\eqref{def1HS} (see \cite{dBR1, dBR2}) the following basic facts listed 
in Theorem \ref{T:H(S)}; the notion of {\em reproducing kernel 
Hilbert space} entering in the first fact is reviewed in Section 
\ref{S:RKHS} below. The proofs of the various pieces of the following 
result are given also in Section \ref{S:RKHS}.

\begin{theorem}  \label{T:H(S)}
    If $S \in \cS(\cU, \cY)$, the space $\cH(S)$ has the following 
    properties:
\begin{enumerate}
    \item  $\cH(S)$ is a linear 
space, indeed a reproducing kernel Hilbert space with reproducing 
kernel $K_{S}(z,w)$ given by
$$
  K_{S}(z,w) = \frac{I - S(z) S(w)^{*}}{1 - z \overline{w}}.
$$
\item The space $\cH(S)$ is invariant under the backward-shift operator
\begin{equation}  \label{RS}
  R_{0} \colon f(z) \mapsto [f(z) - f(0)]/z
\end{equation}
and the following norm estimate holds:
\begin{equation}   \label{norm-est}
    \| R_{0} f\|^{2}_{\cH(S)} \le \| f \|^{2}_{\cH(S)} - \| f(0) 
    \|^{2}_{\cY}.
\end{equation}
Moreover, equality holds in \eqref{norm-est} for all $f \in \cH(S)$ 
if and only if $\cH(S)$ has the property
$$
 S(z) \cdot u \in \cH(S) \Rightarrow S(z) \cdot u \equiv 0.
$$
\item
For any $u \in \cU$, the function $R_0(Su)$
is in $\cH(S)$.  If one lets $\tau \colon \cU \to \cH(S)$ 
denote the operator
\begin{equation}  \label{tau}
  \tau \colon u \mapsto R_0(Su)=\frac{S(z) - S(0)}{z} \,  u, 
\end{equation}
then the adjoint $R_{0}^{*}$ of the operator $R_{0}$ \eqref{RS} on $\cH(S)$ is given by
\begin{equation}   \label{RS*}
  R_{0}^{*} \colon f(z) \mapsto z f(z) - S(z) \cdot \tau^{*}(f)
\end{equation}
with the following formula for the norm holding:
\begin{equation}   \label{RS*norm}
    \| R_{0}^{*} f\|^{2}_{\cH(S)} = \| f \|^{2}_{\cH(S)} - \| 
    \tau^{*}(f) \|^{2}_{\cU}.
\end{equation}

\item
Let  $\bU_{S}$ be the colligation matrix given by
\begin{equation}   \label{bUS}
 \bU_{S} = \begin{bmatrix}  A_{S} & B_{S} \\ C_{S} & D_{S} 
\end{bmatrix}: = \begin{bmatrix} 
     R_{0} & \tau \\ \be(0) & S(0) \end{bmatrix}  \colon 
     \begin{bmatrix} \cH(S) \\ \cU \end{bmatrix} \to \begin{bmatrix} 
	 \cH(S) \\ \cY \end{bmatrix}
\end{equation}
where $R_{0}$ and $\tau$ are given by \eqref{RS} and \eqref{tau} and 
where $\be(0) \colon \cH(S) \to \cY$ is the evaluation-at-zero map:
$$
    \be(0) \colon f(z) \mapsto f(0).
$$
Then $\bU_{S}$ is coisometric, and one recovers $S(z)$ as the characteristic 
function of $\bU_{S}$:
\begin{equation}   \label{realization}
  S(z) = D_{S} + z C_{S} (I - zA_{S})^{-1} B_{S}.
\end{equation}

\item  The operator $T$ on a Hilbert space $\cX$ is unitarily 
equivalent to an operator of the form $R_{0}$ on a model space 
$\cH_{S}$ for a contractive operator-valued function $S$ on ${\mathbb D}$ if and 
only if $T$ is a completely non-isometric contraction, i.e., 
$$
\|T \| \le 1\quad\mbox{and}\quad \bigcap_{n \ge 0} \{ x \colon \| T^{n} x \| = \|x \| \} = 
\{0\}.
$$
\end{enumerate}
\end{theorem}

In addition, there is an extended space $\cD(S)$ constructed as 
follows (see \cite{dBR1, dB70}).  One defines $\cD(S)$ as the space 
of all pairs of functions (here written as columns) $\left[ 
\begin{smallmatrix} f(z) \\ g(z) \end{smallmatrix} \right]$ with 
    $f\in H^{2}(\cY)$ and $g(z) = 
    \sum_{n=0}^{\infty} a_{n} z^{n}\in H^{2}(\cU)$) such that the sequence of 
    numbers
\begin{equation}   \label{Nn}
    N_{n}: = \| z^{n} f(z) - S(z) (a_{0} z^{n-1} + \cdots + a_{n-1}) 
    \|^{2}_{\cH(S)} + \| a_{0}\|^{2} + \cdots + \| a_{n-1}\|^{2} 
\end{equation}
is uniformly bounded.  It can be shown that the sequence 
$\{N_{n}\}_{n\ge 0}$ is in fact nonincreasing so the limit 
${\displaystyle\lim_{n \to \infty} N_{n}}$ exists.  This limit is then defined to be 
the $\cD(S)$-norm of $\left[ \begin{smallmatrix} f(z) \\ g(z) 
\end{smallmatrix} \right]$:
$$
  \left\| \begin{bmatrix} f \\ g \end{bmatrix} 
  \right \|^{2}_{\cD(S)} = \lim_{n \to \infty} N_{n}\quad \text{where $N_{n}$ is as in 
  \eqref{Nn}.}
$$
In particular, if $\left[ \begin{smallmatrix} f \\ g 
\end{smallmatrix} \right] \in \cD(S)$, then necessarily 
$$ 
z^{n} f(z) - S(z) (a_{0} z^{n-1} + \cdots + a_{n-1}) \in \cH(S)
$$ 
for each  $n=0,1,2,\dots$.  The special choice $n=0$ implies that $f \in 
\cH(S)$. The formula \eqref{RS*} for $R_{0}^{*}$  combined with the 
notation $a_{i} = \tau^{*}((R_{0}^{*})^{i+1} f)$ gives rise to the 
formula
$$
z \mapsto z^{n}f(z)  - S(z)(a_{0} z^{n-1} + \cdots + a_{n-1}) = 
(R_{0}^{*})^{n}(f) \in  \cH(S)
$$
for the action of $(R_{0}^{*})^{n}$.
Moreover, the norm identity \eqref{RS*norm} implies that 
$$
\| z^{n}f(z)  - S(z)(a_{0} z^{n-1} + \cdots + 
a_{n-1})\|^{2}_{\cH(S)} + \| a_{n-1}\|^{2} + \cdots + \| a_{0}\|^{2} 
= \| f(z) \|^{2}_{\cH(S)} 
$$
for all $n\ge 0$.  
For this special choice of $g$, namely 
$$
g(z) = \widetilde f(z): = \sum_{n=0}^{\infty} a_{n} z^{n}\quad \text{with}\quad a_{n} 
= \tau^{*}((R_{0}^{*})^{n+1} f),
$$
it follows that $\left[ \begin{smallmatrix} f \\ \widetilde f \end{smallmatrix} \right]
\in \cD(S)$ with $\left\| \left[ \begin{smallmatrix} f \\ \widetilde f \end{smallmatrix} 
\right] 
    \right\|_{\cD(S)}=  \| f \|_{\cH(S)}$.  Thus $f \mapsto \left[ 
    \begin{smallmatrix} f \\ \widetilde f \end{smallmatrix} \right]$ 
	is an isometric embedding of $\cH(S)$ into $\cD(S)$.
	
\smallskip

The following theorem gives the properties of $\cD(S)$ analogous 
to those listed in Theorem \ref{T:H(S)} for $\cH(S)$.  The proofs of 
these results are given in Section \ref{S:D(S)RKHS} below.

\begin{theorem}   \label{T:D(S)}
  Suppose that $S \in \cS(\cU, \cY)$ and the space $\cD(S)$ is 
  defined as above.  Then:
 \begin{enumerate}  
     \item[(1)] $\cD(S)$ is a linear space, indeed a reproducing 
     kernel Hilbert space with reproducing kernel $\widehat 
     K_{S}(z,w)$ given by
  \begin{equation}  \label{DS-ker}
  \widehat K_{S}(z,w) = \begin{bmatrix} K_{S}(z,w) & \frac{S(z) - 
  S(\overline{w})}{z - \overline{w}}  \\
 \frac{\widetilde S(z) - \widetilde S(\overline{w})}{z - \overline{w}} &
  K_{\widetilde S}(z,w) \end{bmatrix}\quad\mbox{where}\quad \widetilde S(z):= 
S(\overline{z})^{*}.
  \end{equation}
  \item[(2)\&(3)] The space $\cD(S)$ is invariant under the transformation 
  $\widehat R_{0}$ given by
  \begin{equation}  \label{hatRS}
  \widehat R_{0} \colon \begin{bmatrix} f(z) \\ g(z) \end{bmatrix} 
  \mapsto \begin{bmatrix} [ f(z) - f(0) ]/z \\ z g(z) - \widetilde 
  S(z) f(0) \end{bmatrix}
  \end{equation}
  with adjoint given by
  \begin{equation}   \label{hatRS*}
  (\widehat R_{0})^{*} \colon \begin{bmatrix} f(z) \\ g(z) 
\end{bmatrix} \mapsto \begin{bmatrix} z f(z) - S(z) g(0) \\
[g(z) - g(0)]/z \end{bmatrix}.
\end{equation}
Moreover, the following norm identities hold:
\begin{align}
    \left\| \widehat R_{0}\left( \begin{bmatrix} f \\ g 
\end{bmatrix} \right)  \right\|^{2} & =  \left\| \begin{bmatrix} f \\ g 
\end{bmatrix} \right\|^{2}_{\cD(S)} - \| f(0) \|^{2}_{\cY}, \notag \\
\left\| (\widehat R_{0})^{*} \left( \begin{bmatrix} f \\ g 
\end{bmatrix} \right)  \right\|^{2} &  =  \left\| \begin{bmatrix} f \\ g 
\end{bmatrix} \right\|^{2}_{\cD(S)} - \| g(0) \|^{2}_{\cU}.
\label{normid's}
\end{align}

\item[(4)] Let $\widehat \bU_{S}$ be the colligation matrix given by
$$
\widehat \bU_{S} = \begin{bmatrix} \widehat A_{S} & \widehat B_{S} \\ 
\widehat C_{S} & \widehat D_{S} \end{bmatrix} \colon \begin{bmatrix} 
\cD(S) \\ \cU \end{bmatrix} \to \begin{bmatrix} \cD(S) \\ \cY 
\end{bmatrix}
$$
where 
\begin{align}
  & \widehat A_{S}: = \widehat R_{0}\vert_{\cD(S)}, 
  \quad 
   \widehat B_{S} \colon u \mapsto \begin{bmatrix} \frac{S(z) - 
  S(0)}{z} u \\ K_{\widetilde S}(z,w) u \end{bmatrix}, \notag \\
  & \widehat C_{S} \colon \begin{bmatrix} f(z) \\ g(z) \end{bmatrix} 
  \mapsto f(0), \quad \widehat D_{S} = S(0).
  \label{hatbU}
\end{align}
Then $\widehat \bU_{S}$ is unitary, and $S$ is recovered as the 
characteristic function of $\widehat \bU_{S}$:
\begin{equation}   \label{realunitary}
  S(z) = \widehat D_{S} + z \widehat C_{S} (I - z \widehat 
  A_{S})^{-1} \widehat B_{S}.
\end{equation}

\item[(5)]  The operator $T$ on a Hilbert space $\cX$ is unitarily 
equivalent to an operator of the form $\widehat R_{0}$ on a model space 
$\cH_{S}$ for 
a contractive operator-valued function $S$ on ${\mathbb D}$ if and 
only if $T$ is a completely non-unitary contraction, i.e., $\|T 
\| \le 1$ and 
$$\left( \bigcap_{n \ge 0} \{ x \colon \| T^{*n} x \| = \|x \| \} 
\right)  \bigcap \left( \bigcap_{n \ge 0} \{ x \colon \| T^{n} x \| = \| 
x \| \} \right) = \{0\}.
$$
\end{enumerate}
    
\end{theorem}

\section{The de Branges-Rovnyak space $\cH(S)$: other formulations} 
\label{S:dBRspaces}

Later, operator theorists, beginning with Douglas \cite{Douglas65} and 
continuing with Fillmore and Williams \cite{FW}, 
Sarason \cite{SarasonBieberbach, sarasonbook}, Ando 
\cite{Ando} and Nikolskii-Vasyunin \cite{NV0, NV1, NV2}, became interested in giving a more 
operator-theoretic formulation for de Branges-Rovnyak spaces leading to better 
insights into the results from the point of view of operator theory; 
there were also unpublished notes of Rosenblum and Douglas 
\cite{Rosenblumnotes, Douglasnotes}.  
To carry this out, one needs a generalization of closed subspace of a 
Hilbert space, namely, contractively included subspace of a Hilbert 
space, and the notion of the complementary space more general than the 
familiar notion of the orthogonal complement for an isometrically 
included closed subspace of a Hilbert space.

\subsection{Lifted-norm spaces}
Suppose that $\cM$ and $\cH$ are Hilbert spaces with $\cM$ a subset 
of $\cH$ but with its own norm $\| \cdot \|_{\cM}$ possibly distinct 
from the norm it inherits from $\cH$ as a subset 
of $\cH$.  The terminology {|em $\cM$ is contractively included in 
$\cH$} shall mean that
the inclusion map $\iota \colon \cM \to \cH$ is contractive, i.e.,
$$
  \|  x \|^{2}_{\cH} \le \| x \|^{2}_{\cM}\quad \text{for all}\quad  x \in \cM.
$$
Then one may define an operator $P$ on $\cH$ by
$P = \iota \iota^{*}$.  Then $P = P^{*}$ and
$$
P^{2} = \iota \iota^{*} \iota \iota^{*} = \iota (\iota^{*} \iota) 
\iota^{*} \le \iota \iota^{*} = P,\quad\mbox{so that}\quad
0 \le P^{2} \le P \le I_{\cH}.
$$  
Conversely, if $P$ is any 
positive semidefinite contraction operator ($0 \le P \le I_{\cH}$), then also $0 
\le P^{2} \le P \le I_{\cH}$ and one may define a Hilbert space $\cM$ 
as $\cM = \operatorname{Ran} P^{\frac{1}{2}}$ with norm given by
$$
   \| P^{\frac{1}{2}}x\|_{\cM} = \|\bQ x \|_{\cH}
$$
(the {\em lifted norm construction})
where $\bQ$ is the orthogonal projection onto $(\operatorname{Ker} 
P)^{\perp} = \overline{\operatorname{Ran}} P$.  Then one can check 
that
$$
\| P^{\frac{1}{2}} x \|_{\cH} = \| P^{\frac{1}{2}} \bQ x \|_{\cH} \le \| \bQ x \|_{\cH} = 
\| P^{\frac{1}{2}} x \|_{\cM},
$$
so $\cM$ is contractively included in $\cH$.  Moreover, the 
computation, for $x = P^{\frac{1}{2}} x_{1} \in \cM$ and $y \in \cH$,
$$
\langle \iota x, y \rangle_{\cH}=\langle \iota P^{\frac{1}{2}} x_{1},
    y \rangle_{\cH}
    = \langle P^{\frac{1}{2}} x_{1}, y \rangle_{\cH} 
 = \langle P^{\frac{1}{2}} x_{1}, P y
    \rangle_{\cM} = \langle x, P y \rangle_{\cM}
$$
shows that 
$$
  \iota^{*} \colon y \in \cH \mapsto P y \in \cM.
$$
Therefore,  $\iota \iota^{*} = P$ as an operator on $\cH$. 
In the sequel the notation $\cM = \cH^{l}_{P}$ (the {\em lifted-norm 
space} associated with the selfadjoint contraction $P$) will be used
whenever the space $\cM$ contractively included in $\cH$ arises in 
this way from the operator $P \in \cL(\cH)$ with $0 \le P \le 
I_{\cH}$.  This discussion leads to the following observation.

\begin{proposition}   \label{P:contrincl}
Contractively included subspaces $\cM$ of a Hilbert space  $\cH$ 
are in one-to-one correspondence with positive semidefinite contraction operators 
$P$ on $\cH$ ($0 \le P^{2} \le P \le I_{\cH}$) according to the formula
$$
   P = \iota \iota^{*}
$$
where $\iota \colon \cM \to \cH$ is the inclusion map, and then
$$
\cM = \operatorname{Ran} P^{\frac{1}{2}}\quad  \text{with}\quad \| 
P^{\frac{1}{2}}g\|_{\cM} = \| \bQ g 
\|_{\cH}
$$
where $\bQ$ is the orthogonal projection of $\cH$ onto 
$\overline{\operatorname{Ran}} P = (\operatorname{Ker} P)^{\perp}$,
written as $\cM = \cH^{l}_{P}$.

The case where $\cM$ is isometrically included in $\cH$ corresponds 
to the case where $P^{2} = P$ and then $P$ is the orthogonal 
projection of $\cH$ onto $\cM$.
\end{proposition}

It is of interest that, even when $P$ is not an orthogonal projection,  
the lifted-norm space $\cH^{l}_{I - P}$  can be viewed as a kind of 
generalized complementary space ({\em Brangesian complement} in the 
terminology of \cite{sarasonbook}) $\cM^{[\perp]}$ to $\cM = 
\cH^{l}_{P}$ as 
explained by the following proposition.

\begin{proposition}   \label{P:Brangesian}
    Let $P \in \cL(\cH)$ with $ 0 \le P  \le I_{\cH}$ and set 
$$
\cM = \cH^{l}_{P}, \quad \cM^{[\perp]} = \cH^{l}_{Q}\quad \text{where}\quad  
Q = I - P.
$$
Then $\cM$ and $\cM^{[\perp]}$ are complementary in the following 
sense:  each $f \in \cH$ has a (not necessarily unique) decomposition
$f = g + h$ with $g \in \cM$ and $h \in \cM^{[\perp]}$.  Moreover, 
the norm of $f$ in $\cH$ is given by
\begin{equation}   \label{inf}
  \| f\|^{2} = \inf \{ \|g \|^{2}_{\cM} + \| h \|^{2}_{\cM^{[\perp]}} 
  \colon g \in \cM \text{ and } h \in \cM^{[\perp]} \text{ such that 
  } f = g + h \}.
\end{equation}
Moreover:
\begin{enumerate}
    \item The infimum in \eqref{inf} is attained when $g = Pf \in 
\cH^{l}_{P}$ and $h = Q f \in \cH^{l}_{Q}$.

\item The space $\cM^{[\perp]} = \cH^{l}_{Q}$ ($Q = I-P$) can alternatively be 
characterized as
\begin{equation}   \label{Mperp}
\cM^{[\perp]} = \left\{ h \in \cH \colon \| h \|^{2}_{Q}: = \sup\{ \| g + 
 h\|^{2}_{\cH} - \|g \|^{2}_{\cH^{l}_{P}} \colon g \in \cH^{l}_{P} \} < 
\infty\right \}
\end{equation}
and then $\| h \|_{\cH^{l}_{Q}} = \| h \|_{Q}$.
\end{enumerate}
\end{proposition}

\begin{proof}
Note first that, since $P + Q = I_{\cH}$ by definition,
    any $f \in \cH$ has a decomposition $f = Pf + Qf$ where $g = P f 
    \in \cH^{l}_{P}$ and $h = Q f \in \cH^{l}_{Q}$.   
    
\smallskip

    Next assume that $f = g + h$ with $g \in \cH^{l}_{P}$ 
    and $h \in \cH^{l}_{Q}$.  By Proposition
\ref{P:contrincl} one can find $g_{1} \in
\overline{\operatorname{Ran}} P$ and $h_{1} \in
\overline{\operatorname{Ran}} Q$   so that
\begin{equation} \label{need}
g = P^{\frac{1}{2}} g_{1}, \quad h = Q^{\frac{1}{2}} h_{1},\quad
(g_{1} \in \overline{\operatorname{Ran}} P).
\end{equation}
Taking into account that 
$P+Q = I_{\cH}$ and that $P^{\frac{1}{2}}
Q^{\frac{1}{2}} = Q^{\frac{1}{2}} P^{\frac{1}{2}}$, one then computes
\begin{align}
    & \| g \|^{2}_{\cH^{l}_{P}} + \| h \|^{2}_{\cH^{l}_{Q}} - \| f 
    \|^{2}_{\cH} = \| g_{1} \|^{2}_{\cH} + \| h_{1} \|^{2}_{\cH} - \| 
    P^{\frac{1}{2}}g_{1} + Q^{\frac{1}{2}} h_{1}\|^{2}_{\cH}  \notag\\
    & \quad = \| g_{1} \|^{2}_{\cH} + \| h_{1} \|^{2}_{\cH} - \langle P 
    g_{1}, g_{1} \rangle_{\cH} - 2 \operatorname{Re} \langle P^{\frac{1}{2}} 
    g_{1}, Q^{\frac{1}{2}} h_{1} \rangle_{\cH} - \langle Q h_{1}, h_{1} 
    \rangle_{\cH}  \notag\\
    & \quad = \langle Q g_{1}, g_{1} \rangle_{\cH} + \langle P h_{1}, 
    h_{1} \rangle_{\cH} - 2 \operatorname{Re} \langle Q^{\frac{1}{2}} g_{1}, 
    P^{\frac{1}{2}} h_{1} \rangle_{\cH}\notag  \\
    & \quad = \| Q^{\frac{1}{2}} g_{1} - P^{\frac{1}{2}} h_{1} \|_{\cH}^{2} \ge 
0\label{need1}
\end{align}
with equality if and only if  $Q^{\frac{1}{2}} g_{1} = P^{\frac{1}{2}} h_{1}$.

\smallskip

To check property (1), note that $g = Pf \Rightarrow g_{1} = P^{\frac{1}{2}} 
f$ and $h = Qf \Rightarrow h_{1} = Q^{\frac{1}{2}} f$. Therefore,
$$
Q^{\frac{1}{2}} g_{1} = Q^{\frac{1}{2}} P^{\frac{1}{2}} f = P^{\frac{1}{2}} 
Q^{\frac{1}{2}} f = P^{\frac{1}{2}} h_{1}
$$
and hence equality occurs in \eqref{inf} with this choice of 
$g$ and $h$.
Uniqueness follows from the general fact that closed convex sets in a 
Hilbert space have a unique element of minimal norm.  

\smallskip

It remains only to verify statement (2) in the proposition.
Given $h \in \cH^{l}_{Q}$, define $\|h\|_{Q}$ as in condition \eqref{Mperp}:
$$
\|h\|^{2}_{Q} = \sup_{g \in \cH^{l}_{P}} \{ \| g + h \|^{2}_{\cH} - \| g 
\|^{2}_{\cH^{l}_{P}} \}.
$$
It has been already shown that, for any $g \in \cH^{l}_{P}$,
$$
  \| g+h\|^{2}_{\cH} \le \| g \|^{2}_{\cH^{l}_{P}} + \| h 
  \|^{2}_{\cH^{l}_{P}},
$$
from which it follows that
\begin{equation}   \label{ineq1}
  \|h\|^{2}_{\cH^{l}_{Q}} \ge \| h \|^{2}_{Q}. 
\end{equation}
The following is an alternative direct proof of \eqref{ineq1} which provides 
some additional information which will be needed later. Take $h \in 
\cH^{l}_{Q}$ and $g \in \cH^{l}_{P}$ in the form \eqref{need}.
Computation \eqref{need1} gives 
\begin{align}
    & \| h \|^{2}_{\cH^{l}_{Q}} - \| g + h \|^{2}_{\cH} + \| g 
    \|^{2}_{\cH^{l}_{P}}\notag\\
&\quad = 
 \| Q^{\frac{1}{2}} h_{1} \|^{2}_{\cH^{l}_{Q}} - \| P^{\frac{1}{2}} g_{1} + 
Q^{\frac{1}{2}} 
 h_{1} \|^{2}_{\cH} + \| P^{\frac{1}{2}} g_{1} \|^{2}_{\cH^{l}_{P}} \notag \\
& \quad = \| P^{\frac{1}{2}}h_{1} - Q^{\frac{1}{2}} g_{1} \|_{\cH}^{2} \ge 0 \label{ineq2}
\end{align}
from which \eqref{ineq1} follows.

\smallskip

Suppose now that $h \in \cH$ with $\| h \|_{Q}< \infty$.  It suffices to  
show that $h \in \cH^{l}_{Q} = \operatorname{Ran} Q^{\frac{1}{2}}$ and that 
the reverse inequality
\begin{equation}   \label{ineq3}
    \| h\|^{2}_{\cH^{l}_{Q}} \le \| h \|^{2}_{Q}
\end{equation}
holds.  From the fact that $\| h\|^{2}_{Q}< \infty$ one can see that 
$$
\| h \|^{2}_{\cH} + 2 \operatorname{Re} \langle h, P^{\frac{1}{2}} g_{1} 
\rangle_{\cH} + \| P^{\frac{1}{2}} g_{1} \|^{2} \le M + \| g_{1} \|^{2}
$$
for all $g_{1} \in \overline{\operatorname{Ran}} P$ for some constant 
$M < \infty$.  It thus follows that
\begin{align}
 0 & \le  M - \langle h, h \rangle_{\cH} - 2 \operatorname{Re} \langle h, 
  P^{\frac{1}{2}} g_{1} \rangle_{\cH} + \| g_{1} \|^{2} - \| P^{\frac{1}{2}} g_{1} 
  \|^{2} \notag  \\
  & = M - \|h\|^2_{\cH} - 2 \operatorname{Re} \langle h, 
  P^{\frac{1}{2}} g_{1} \rangle_{\cH} + \langle (I-P) g_{1}, g_{1} 
  \rangle_{\cH} \notag \\
  & = M - \|h\|^2_{\cH}- 2 \operatorname{Re} \langle h, 
  P^{\frac{1}{2}} g_{1} \rangle_{\cH} + \|Q^{\frac{1}{2}}g_{1}\|^2_{\cH}
  \label{ineq4}
\end{align}
for all $g_{1} \in \overline{\operatorname{Ran}} P$.  Set
$$
  M_{1}: = M - \|h\|^2_{\cH}.
$$
Then $M_{1} \ge 0$ since one may choose $g_{1} = 0$ in the inequality 
\eqref{ineq4}.  Replacing $g_{1}$ by $\omega t g_{1}$ where $t$ is an 
arbitrary real number and $\omega$ is an appropriate unimodular 
constant, \eqref{ineq4} may be rewritten in the form
\begin{equation}   \label{ineq5}
 M_{1} - 2 | \langle h, P^{\frac{1}{2}} g_{1} \rangle_{\cH}| t + 
\|Q^{\frac{1}{2}}g_{1}\|^2_{\cH}t^{2} \ge 0 \quad \text{for 
 all real \; } t.
\end{equation}
The Quadratic Formula test for the roots of a real polynomial implies that  
$$ | \langle h, P^{\frac{1}{2}} g_{1} \rangle_{\cH} |^{2} \le M_{1} \| Q^{\frac{1}{2}} 
g_{1} \|^{2}.
$$
The Riesz representation theorem for a linear functional on a 
Hilbert space then implies that there is an $\widetilde h \in 
\overline{\operatorname{Ran}} Q$ so 
that 
$$
  \langle h, P^{\frac{1}{2}} g_{1} \rangle_{\cH} = \langle \widetilde 
 h, Q^{\frac{1}{2}} g_{1} \rangle_{\cH}.
$$
It now follows from this last identity  that $P^{\frac{1}{2}} h = Q^{\frac{1}{2}} \widetilde h$ and
\begin{equation}
h = P h + Q h = P^{\frac{1}{2}} Q^{\frac{1}{2}} \widetilde h + Q h=  
Q^{\frac{1}{2}}P^{\frac{1}{2}} \widetilde h + Q h= Q^{\frac{1}{2}} h_{1}
\label{Rieszrep}
\end{equation}
with $h_{1} = P^{\frac{1}{2}} \widetilde h +Q^{\frac{1}{2}} h$, 
which implies that $h \in \cH^{l}_{Q}$. Furthermore, for $h_1$ 
as defined above, 
$$
  P^{\frac{1}{2}} h_{1} = P^{\frac{1}{2}}(P^{\frac{1}{2}} \widetilde h + Q^{\frac{1}{2}} 
h) = P
  \widetilde h + P^{\frac{1}{2}} Q^{\frac{1}{2}} h = (I - Q) \widetilde h + 
Q^{\frac{1}{2}}
  P^{\frac{1}{2}} h
$$
is in $\overline{\operatorname{Ran}} Q$, $\widetilde h$ itself was arranged 
to be in $\overline{\operatorname{Ran}} Q$.
By Proposition \ref{P:contrincl}, any $g \in 
\cH^{l}_{P}$ can be written as $g = P^{\frac{1}{2}} g_{1}$ with $g_{1} \in 
\overline{\operatorname{Ran}}  P$. For this arbitrary $g$ and for $h = Q^{\frac{1}{2}} 
h_{1} \in \cH^{l}_{Q}$
(see \eqref{ineq2}), it holds that
\begin{equation}
 \| g+h\|^{2}_{\cH} - \| g \|^{2}_{\cH^{l}_{P}} = \| h 
 \|^{2}_{\cH^{l}_{Q}} - \| P^{\frac{1}{2}} h_{1} - Q^{\frac{1}{2}} g_{1} \|_{\cH}^{2}.
\label{need2}
\end{equation}
Since $P^{\frac{1}{2}} h_{1}$ is in $\overline{\operatorname{Ran}} Q$, it follows that
$$
  \inf_{g_{1} \in \cH} \| P^{\frac{1}{2}} h_{1} - Q^{\frac{1}{2}} g_{1} \|= 0.
$$
Since $P$ and $Q$ commute, one even has
$$
 \inf_{g_{1} \in
  \overline{\operatorname{Ran}} P }\| P^{\frac{1}{2}} h_{1} - Q^{\frac{1}{2}} g_{1} \| = 
0.
$$
Combining this with \eqref{need2} and \eqref{Mperp} leads to the reverse inequality 
\eqref{ineq3}, and completes the verification of statement (2) in 
Proposition \ref{P:Brangesian}.
\end{proof}

Note next that if $\cM = \cH^{l}_{P}$, then  
$\cM^{[\perp]} = \cH^{l}_{Q}$ with $Q = I-P$.  Hence for the 
complementary space $(\cM^{[\perp]})^{[\perp]}$ of $\cM^{[\perp]}$, 
$$
 ( \cM^{[\perp]})^{[\perp]} = \cH^{l}_{I-Q} = \cH^{l}_{P} = \cM,
$$
i.e., one comes back to $\cM$ itself.  The following corollary is 
immediate from this observation combined with Proposition 
\ref{P:Brangesian}.

\begin{corollary}   \label{C:Brangesian}  
  In addition to the complementary space $\cM^{[\perp]}$ being 
  recovered from $\cM$ via the criterion \eqref{Mperp}, one can also 
  recover $\cM = \cH^{l}_{P}$ from $\cM^{[\perp]} = \cH^{l}_{Q}$ ($Q 
  = I - P$) in the same way:
$$ 
   \cM = \{ g \in \cH \colon \| g \|^{2}_{P}: = 
\sup_{h \in \cH^{l}_{Q}}\| g + 
 h\|^{2}_{\cH} - \|h \|^{2}_{\cH^{l}_{Q}} < \infty\}.
$$
\end{corollary}

The next Proposition presents the role of the {\em overlapping space} 
in measuring the extent to which the Brangesian complementary space 
fails to be a true orthogonal complement.

\begin{proposition}  \label{P:overlapping}
The map $\Xi \colon f \oplus g \mapsto f+g$ is a partial isometry from 
$\cM \oplus \cM^{[\perp]}$ onto $\cH$.
Furthermore, if one introduces the {\em overlapping space} $\cL_{P \cdot Q}$ 
by 
$$
\cL_{P \cdot Q} = \cM^{[\perp]} \cap \cM = \cH^{l}_{P} \cap
\cH^{l}_{Q}\quad\mbox{with norm}\quad 
  \| f \|^{2}_{\cL_{P \cdot Q}} = \| f \|^{2}_{\cH^{l}_{P}} + \| f
  \|^{2}_{\cH^{p}_{Q}},
$$
then the kernel of the linear transformation $\Xi \colon 
\cM^{[\perp]} \oplus \cM \to \cH$ is given by
$$
  \operatorname{Ker} \Xi = \{ f \oplus -f \in \cM \oplus 
  \cM^{[\perp]} \colon f \in \cL_{P \cdot Q} \}
$$
and the map $\widehat \Xi \colon \cM^{[\perp]} \oplus \cM \to \cH 
\oplus \cL_{T}$ given by
$$
    \widehat \Xi \colon f \oplus g \mapsto (f + g) \oplus k\quad 
\text{where}\quad k \oplus -k = P_{\operatorname{Ker} \Xi}(f \oplus g)
$$
is unitary.
\end{proposition}

\begin{proof}
    This follows essentially from the definitions.
 \end{proof}

\subsection{Pullback spaces}  \label{S:pullback}

The following variant of the lifted-norm construction given above will be useful in the 
sequel. 
Let $T \in \cL(\cH_{0}, \cH)$ be any contraction operator between 
two Hilbert spaces $\cH_{0}$ and $\cH$ (in particular, even if 
$\cH_{0} = \cH$, $T$ is not necessarily 
positive or even selfadjoint) and set $\cM = \operatorname{Ran} T$ 
with norm given by
\begin{equation}   \label{p-norm}
   \| T x \|_{\cM} = \| \bQ x \|_{\cH}
\end{equation}
where $\bQ$ is the orthogonal projection of $\cH_{0}$ onto 
$(\operatorname{Ker} T)^{\perp} = \overline{\operatorname{Ran}} T^{*}$ 
(the {\em pull-back construction}). As $T$ is an isometry from 
the complete space $\operatorname{Ran} \bQ$ in the $\cH_{0}$-norm onto 
$\cM$, it is easily seen that $\cM$ so defined is a Hilbert space.
Whenever the Hilbert space $\cM$ contractively included in $\cH$ has 
the form $\cM = \operatorname{Ran T}$ for a contraction operator $T$ 
with norm given by \eqref{p-norm}, the notation $\cM = \cH^{p}_{T}$ 
(the {\em pull-back} space associated with $T$) shall be applied.

\smallskip

Suppose that $\cM = \cH^{p}_{T}$ and let $\iota \colon \cM = \operatorname{Ran} T \to \cH$
be the inclusion map.  The following computation
\begin{align*}
    \langle \iota T x, y \rangle_{\cH} & =
 \langle Tx, y \rangle_{\cH} = \langle x, T^{*} y \rangle_{\cH} = 
 \langle Tx, T T^{*} y \rangle_{\cM}
\end{align*}
shows that 
$$
\iota^{*} \colon y \in \cH \mapsto T T^{*} y \in \cM
$$
and hence $\iota \iota^{*} = T T^{*} = :P$ as an operator on $\cH$.  
Therefore the pull-back space $\cH^{p}_{T}$ is isometrically 
equal to the lifted norm space $\cH^{l}_{T T^{*}}$, and the 
lifted-norm space $\cH^{l}_{P}$ (where $0 \le P \le I$) is 
isometrically equal to the pull-back space $\cH^{p}_{P^{1/2}}$.

\smallskip

While a lifted norm space uniquely determines the associated positive 
contraction $P$ ($\cH^{l}_{P} = \cH^{l}_{P'} \Leftrightarrow P = P'$), pullback 
spaces determine the associated contraction operator only up to a 
partially isometric right factor: {\em $\cH^{p}_{T} = \cH^{p}_{T'}$ if and 
only if there is a partial isometry $\alpha \colon \cH_{0} \to 
\cH_{0}'$ so that $T = T' \alpha$ and $T' = T \alpha^{*}$}.

\smallskip
In conclusion, it follows that all the observations made in the previous section 
concerning 
lifted-norm spaces apply equally well to pullback spaces. In 
general, the Brangesian complementary spaces 
$(\cH^{p}_{T})^{[\perp]}$ to the pullback space $\cH^{p}_{T}$ can be 
identified with the lifted-norm space $\cH^{l}_{I - T T^{*}}$ (or 
equivalently, the pullback space $\cH^{p}_{(I - T T^{*})^{\frac{1}{2}}}$).
An immediate consequence of these observations and Corollary 
\ref{C:Brangesian} is:  {\em given a contraction 
operator $T \in \cL(\cH_{0}, \cH)$, the space $\cH^{p}_{T}$ can be 
characterized as}
$$
\cH^{p}_{T} = \{ h \in \cH \colon \sup_{g \in \cH^{l}_{I - T
T^{*}}} \{ \| g + h \|^{2}_{\cH} - \| 
g \|^{2}_{\cH^{l}_{I - T T^{*}}}\} < \infty\}
$$
(see also \cite[Theorem 4.1]{FW}).
Moreover, the pullback spaces $\cH^{p}_{T}$ is isometrically included in $\cH$ 
exactly when $T T^{*}$ is a projection, i.e., when $T$ is a partial isometry.  

\smallskip

The overlapping space $\cL_{P \cdot Q}$ construction in Proposition 
\ref{P:overlapping} has a slightly different form in the original de 
Branges-Rovnyak theory \cite{dBR1, dBR2} for the special case where 
$P = I - TT^{*}$ and $Q = T T^{*}$ which will be now  described.  In 
the case where $T$ is an isometry (not just a partial isometry as 
came up in the previous paragraph), then the operator
$$
\begin{bmatrix} \iota _{\cH^{l}_{I - t T^{*}}} & T \end{bmatrix}
\colon f \oplus g \mapsto f + T g
$$
is unitary from $\cH^{l}_{I - T T^{*}} \oplus \cH_{0}$ onto $\cH$.
The de Branges-Rovnyak overlapping space $\cL_{T}$ measures the 
extent to which $\begin{bmatrix} \iota_{\cH^{l}_{I - T T^{ *}}} & 
I_{\cH_{0}} \end{bmatrix}$ fails to be isometric for the case of a 
general contraction operator $T \in \cL(\cH_{0}, \cH)$.
Define the space $\cL_{T}$ by
\begin{equation}  \label{cLT1}
   \cL_{T} = \left\{ f \in \cH_{0} \colon  T f \in \cH^{l}_{I - T 
   T^{*}} \right\}
\end{equation}
with norm given by
\begin{equation}   \label{cLT2}
   \| f \|^{2}_{\cL_{T}} = \| T f \|^{2}_{\cH^{l}_{I - T T^{*}}} + \| f 
   \|^{2}_{\cH_{0}}.
\end{equation}

\begin{proposition}   \label{P:cLT}
    For a contraction operator  $T\in\cL(\cH_{0}, \cH)$ define the 
    overlapping space $\cL_{T}$ via \eqref{cLT1}, \eqref{cLT2}. 
Let $\Xi_{T} \colon \cH^{l}_{I - T T^{*}} \oplus \cH_{0} \to \cH$ 
be the operator given by
$$
  \Xi_{T} = \begin{bmatrix} \iota_{\cH^{l}_{I - T T^{*}}} & T 
\end{bmatrix} \colon \begin{bmatrix}  f \\ g \end{bmatrix} \mapsto f 
+ T g.
$$
Then $\Xi_{T}$ is a coisometry from $\cH^{l}_{I - T T^{*}} \oplus 
\cH_{0}$ onto $\cH$ with kernel given by
$$
  \operatorname{Ker} \Xi_{T} = \left\{ \begin{bmatrix} Tf \\ -f 
\end{bmatrix} \colon f \in \cL_{T} \right\},
$$
and the map  
$$ \widehat \Xi_{T} \colon \begin{bmatrix} f \\ g \end{bmatrix} \mapsto 
 \begin{bmatrix} f + T g \\ h \end{bmatrix} \text{ \rm where } h \in 
     \cL_{T} \text{ \rm is determined by } \begin{bmatrix} Th \\ -h 
 \end{bmatrix} = P_{\operatorname{Ker} \Xi_{T}} \begin{bmatrix} f \\ 
 g \end{bmatrix}
$$
is unitary from $\cH^{l}_{I - T T^{*}} \oplus \cH_{0}$ onto $\cH 
\oplus \cL_{T}$. Moreover, the overlapping space $\cL_{T}$ is itself 
isometrically equal to a 
lifted norm space:
\begin{equation}   \label{isomiden}
   \cL_{T} = \cH^{l}_{I - T^{*}T}.
\end{equation}
\end{proposition}

\begin{proof}  By definition $T$ is a coisometry from $\cH_{0}$ onto 
    $\cH^{p}_{T} = \cH^{l}_{T T^{*}}$.  By Proposition 
    \ref{P:overlapping}, the map $\Xi^{l}_{TT^{*}} = \begin{bmatrix} 
    \iota_{\cH^{l}_{I - T T^{*}}} & \iota_{\cH^{l}_{TT^{*}}} 
\end{bmatrix}$ is a coisometry from $\cH^{l}_{I - T T^{*}} \oplus 
\cH^{l}_{TT^{*}}$ onto $\cH$.  Note next that the factorization
\begin{equation}   \label{fact}
 \Xi_{T} = \Xi^{l}_{TT^{*}} \circ \begin{bmatrix} 
 I_{\cH^{l}_{TT^{*}}} & 0 \\ 0 & T \end{bmatrix}
\end{equation}
exhibits the map $\Xi_{T}$ as the composition of coisometries (here 
$T$ is viewed as an element of $\cL(\cH_{0}, \cH^{p}_{T})$) and hence 
$\Xi_{T}$ is a coisometry as asserted.  

\smallskip

Proposition 
\ref{P:overlapping} identifies $\operatorname{Ker} \Xi^{l}_{TT*}$ as 
$\{ f \oplus -f \colon f \in \cL_{P \cdot Q}\}$ (where $P = I 
- T T^{*}$ and $Q = T T^{*}$).  From the factorization \eqref{fact} 
one can see that 
$\left[ \begin{smallmatrix} f \\ g \end{smallmatrix} \right]\in 
\operatorname{Ker} \Xi_{T}$ if and only if $\left[ 
\begin{smallmatrix} f \\ T g \end{smallmatrix} \right] \in \operatorname{Ker} 
    \Xi^{l}_{TT^{*}}$.  By Proposition \ref{P:overlapping}, this 
    means that  $\left[ \begin{smallmatrix} f \\ T g 
\end{smallmatrix} \right]$ has the form $\left[ \begin{smallmatrix} f \\ -f 
\end{smallmatrix} \right]$ with $f \in \cL_{P \cdot Q} = 
\cH^{l}_{I-TT^{*}} \cap \cH^{p}_{T}$.  Thus $Tg = -f \in \cH^{l}_{I - 
T T^{*}} \cap \cH^{p}_{T}$ so $g \in \cL_{T}$ and $\left[ 
\begin{smallmatrix} -Tg \\ g \end{smallmatrix} \right] \in 
    \operatorname{Ker} \Xi_{T}$.  The unitary property of $\widehat 
    \Xi_{T}$ now follows easily.

\smallskip
    
    It remains only to verify that $\cL_{T} = \cH^{l}_{I - T^{*}T}$ 
    isometrically.  Suppose first that $g = (I - T^{*} T)^{\frac{1}{2}} g_{1} 
    \in \cH^{l}_{I - T^{*}T}$.  Then certainly $g \in \cH_{0}$.  But 
    also,
 $$
 T g = T (I - T^{*} T)^{\frac{1}{2}} g_{1} = (I - T T^{*})^{\frac{1}{2}} T g_{1} \in 
 \cH^{l}_{I - T T^{*}}.
 $$
 Moreover, the same intertwining $T(I-T^{*}T) = (I - T 
 T^{*}) T$ implies that $T g_{1} \in \overline{\operatorname{Ran}} 
 \cH^{l}_{I-TT^{*}}$ since $g_{1} \in \overline{\operatorname{Ran}} 
 (I - T^{*} T)$. Therefore,
\begin{align*}
    \| g \|^{2}_{\cL_{T}} & = \| h \|^{2}_{\cH_{0}} + \| T g 
    \|^{2}_{\cH^{l}_{I - T T^{*}}} \\
  & = \| (I - T^{*} T)^{\frac{1}{2}} g_{1} \|^{2}_{\cH_{0}} + 
 \| (I - T T^{*})^{\frac{1}{2}} T g_{1} \|^{2}_{\cH^{l}_{I - TT^{*}}} \\
 & = \langle (I - T^{*} T) g_{1}, g_{1} \rangle_{\cH_0} + \| T 
 g_{1}\|^{2}_{\cH} 
= \| g_{1} \|^{2}_{\cH_{0}}= \| g \|^{2}_{\cH^{l}_{I - T^{*} T}}
\end{align*}
and the equality of norms follows.
Conversely, if $g \in \cL_{T}$, then it follows that $g \in 
\cH_{0}$ with $Tg \in \cH^{l}_{I - T^{*}T}$.  Hence there is a 
$\widetilde g \in \overline{\operatorname{Ran}} (I - T T^{*})$ so that
$Tg = (I - T T^{*})^{\frac{1}{2}} \widetilde g$. Therefore,
\begin{align*}
g & = T^{*} T g + (I - T^{*} T) g \\
& = T^{*} (I - T T^{*})^{\frac{1}{2}} \widetilde g + (I - T^{*} T) g \\
& = (I - T^{*} T)^{\frac{1}{2}} T^{*} \widetilde g + (I - T^{*} T) g \in 
\operatorname{Ran} (I - T^{*} T)^{\frac{1}{2}}
\end{align*}
which allows to conclude that $g \in \cH^{l}_{I - T^{*} T}$.  
The isometric equality \eqref{isomiden} has now been verified.
\end{proof}

\subsection{Spaces associated with Toeplitz operators}   
\label{S:Toeplitz}
From now on it will be assumed that all Hilbert spaces are separable.  
For $\cU$ a coefficient Hilbert space, let $L^{2}(\cU)$ denote the 
Hilbert space of weakly measurable norm-square integrable functions 
on the unit circle ${\mathbb T}$; in terms of Fourier series 
representation, one can write
$$ 
L^{2}(\cU) = \left\{ f(\zeta) \sim \sum_{n=-\infty}^{\infty} f_{n} 
\zeta^{n} \colon f_{n} \in \cU \text{ with } \| f \|^{2}_{L^{2}(\cU)} 
: = \sum_{n=-\infty}^{\infty} \| f_{n} \|^{2} < \infty \right\}.
$$
The vector-valued Hardy space $H^{2}(\cU)$ is the subspace of 
$L^{2}(\cU)$ consisting of functions $f(\zeta) = \sum_{n=0}^{\infty} 
f_{n} \zeta^{n}$ having $f_{n} = 0$ for $n<0$ and can also be viewed 
as the space of $\cU$-valued analytic functions on the unit disk ${\mathbb 
D}$ 
having $L^{2}$-norm along circles of radius $r$ uniformly bounded as 
$r \uparrow 1$.  Given two coefficient Hilbert spaces $\cU$ and 
$\cY$, let $L^{\infty}(\cU, \cY)$ denote the space of weakly 
measurable essentially bounded $\cL(\cU, \cY)$-valued functions on 
${\mathbb T}$ ($W \colon {\mathbb T} \to \cL(\cU, 
\cY)$).  Given $W \in L^{\infty}(\cU, \cY)$, let $L_{W} \colon 
L^{2}(\cU) \to L^{2}(\cY)$ denote the {\em Laurent operator} of 
multiplication by $W$ on vector-valued $L^{2}$:
$$
  L_{W} \colon f(\zeta) \mapsto W(\zeta) f(\zeta).
$$
The Toeplitz operator $T_{W}$ associated with $W$ is the compression 
of $L_{W}$ to the Hardy space:
$$
  T_{W} \colon f \mapsto P_{H^{2}(\cY)}( L_{W} f)\quad  \text{for}\quad
 f \in  H^{2}(\cU).
$$
Let $H^{\infty}(\cU, \cY)$ denote the subspace of 
$L^{\infty}(\cU, \cY)$ consisting of $W$ with negative Fourier 
coefficients vanishing:  $W(\zeta) \sim \sum_{n=0}^{\infty} W_{n} 
\zeta^{n}$; as in the vector-valued case, $W(\zeta)$ can be viewed as 
the almost everywhere existing nontangential weak-limit boundary value function 
of an operator-valued function $z \mapsto W(z)$ on the unit disk 
${\mathbb D}$ (here the separability assumption on the coefficient 
Hilbert spaces is invoked---see e.g.~\cite{RR} for details).  
For the case of $W \in H^{\infty}(\cU, \cY)$, 
the Toeplitz operator $T_{W}$ assumes the simpler form
$$
  T_{W} \colon f(\zeta) \mapsto W(\zeta) \cdot f(\zeta).
$$
In this case one says that $W$ is an {\em analytic Toeplitz operator} 
(see \cite{RR}).

\smallskip

The de Branges-Rovnyak spaces discussed in 
\cite[Appendix]{dBR1} and \cite{dBR2} amount to the special case of 
the constructions in Section \ref{S:pullback} above applied to the case where 
$\cH_{0} = H^{2}(\cU)$, $\cH = H^{2}(\cY)$ and $T$ is the analytic 
Toeplitz operator $T = T_{S}$.

\smallskip

An easy consequence of the characterization of uniqueness discussed above for 
pullback spaces is the following:  {\em two Schur-class functions $S \in \cS(\cU, \cY)$ 
and $S \in \cS(\cU', \cY)$ determine the same pullback space $\cM(S) = 
\cH^{p}_{T_{S}} = \cH^{p}_{T_{S'}}$ (and hence also the same de 
Branges-Rovnyak space $\cH(S) = \cH(S')$) if and only if there is a 
partially isometric multiplier $\alpha$ so that $S = S'\alpha$ and 
$S' = S \alpha^{*}$.} In particular, there is a choice of 
Beurling-Lax representer $S$ for a given $\cM = \cH^{p}_{T_{S}}$   with 
the additional property that
\begin{equation}   \label{normalize}
    \{ u \in \cU \colon S(z) u \equiv 0\} = \{0\}.
\end{equation}

In detail, the following 
identification of the de Branges-Rovnyak space $\cH(S)$ as a 
lifted-norm space holds.

\begin{proposition}   \label{P:dBR}
    For an $S\in\cS(\cU, \cY)$,  let 
    $\cH(S)$ be the de Branges-Rovnyak space as defined by  
    \eqref{def1HS} above. 
    Then $\cH(S)$ is isometrically equal to the lifted norm space
    \begin{equation}   \label{def2HS}
      \cH(S) = \cH^{l}_{I-T_{S} T_{S}^{*}}.
    \end{equation}
    Equivalently, if $\cM(S)$ denotes the pullback space
    $$
      \cM(S): = \cH^{p}_{T_{S}},
    $$
    then $\cH(S)$ is the Brangesian complementary space 
    $(\cM(S))^{[\perp]}$ to $\cM(S)$ in $H^{2}(\cY)$.
    \end{proposition}
    
    \begin{proof}  All this can be seen from the definition of the 
	$\cH(S)$ norm in \eqref{def1HS} combined with Proposition 
	\ref{P:Brangesian} and the equivalence between pullback 
	spaces and lifted norm spaces as explained in Section 
	\ref{S:pullback}.
\end{proof}

The next result indicates how one can get parts (2) and (3) in Theorem 
\ref{T:H(S)} using the lifted-norm characterization of $\cH(S)$.  The 
key tool for this task is the following fundamental result of Douglas.

\begin{proposition}  \label{P:Douglas} (See \cite{Douglas65}.) Given two Hilbert space 
    operators $A \in \cL(\cH_{1}, \cH_{2})$ and $B \in \cL(\cH_{0}, 
    \cH_{2})$, there exists a contraction operator $X \in \cL(\cH_{0}, 
    \cH_{1})$ with $AX=B$ if and only if the operator inequality
    $ BB^{*} \le A A^{*}$ holds.
\end{proposition}

For $\cX$ any coefficient Hilbert space, let $\bS_{\cX}$ denote the 
unilateral shift operator of multiplicity $\operatorname{dim} \cX$ as 
modeled on $H^{2}(\cX)$:
$$
   \bS_{\cX} \colon f(z) \mapsto z f(z)\quad \text{for}\quad f \in 
H^{2}(\cX).
$$
Then it is easily verified that its adjoint is given by the 
difference-quotient transformation:
$$
\bS_{\cX}^{*} \colon f(z) \mapsto \frac{f(z) - f(0)}{z} 
\quad \text{for}\quad f \in H^{2}(\cX).
$$
Now part (2) (apart from an analysis of when equality holds 
in\eqref{norm-est} which will come later) and part (3) of Theorem \ref{T:H(S)} 
can be verified as follows.

\begin{theorem}   \label{T:H(S)23}
    Let $S$ be a Schur-class operator-valued function in $\cS(\cU, 
    \cY)$.  Then:
 \begin{enumerate}
     \item[(2)]  The space $\cH(S)$ is invariant under the 
     difference-quotient transformation $\bS_{\cY}^{*}$ with 
     \begin{equation}  \label{normest}
       \|\bS_{\cY}^{*} f\|^{2}_{\cH(S)} \le \| f \|^{2}_{\cH(S)} - 
       \|f(0)\|^{2}_{\cH(S)}.
     \end{equation}

 \item[(3)]  For any vector $u \in \cU$, the function $\frac{ S(z) - 
 S(0)}{z} u$ belongs to $\cH(S)$.  Let $R_{0} \colon \cH(S) \to 
 \cH(S)$ and $\tau  \colon \cU \to \cH(S)$  be the operators $R_{0} = 
\bS_{\cY}^{*}|_{\cH(S)}$ and 
$\tau \colon u \mapsto  \frac{S(z) - S(0)}{z} u$. Then $R_{0}^{*} \in \cL(\cH(S))$ is 
given by
 \begin{equation}   \label{RS*1}
R_{0}^{*} \colon f(z) \mapsto z f(z) - S(z) \cdot \tau^{*}(f)
 \end{equation}
 with the norm of $\tau^{*}(f)$ given by
 \begin{equation}   \label{normtau*}
 \| \tau^{*}(f)\|^{2}_{\cU} = \| f \|^{2}_{\cH(S)} - \| R_{0}^{*}f\|^{}_{\cH(S)}.
 \end{equation}
\end{enumerate}
\end{theorem}

The following proof  synthesizes arguments from \cite{NV0} and 
     \cite{SarasonBieberbach}.

\begin{proof}[Proof of (2).]
View $\cH(S)$ as $\operatorname{Ran} (I - T_{S}T_{S}^{*})^{\frac{1}{2}}$ and  
 introduce the notation
$$
  P_{0} = I - \bS_{\cY}\bS_{\cY}^{*}
$$
for the projection onto the constant functions in $H^{2}(\cY)$.
Next observe the identity
\begin{align*}
 & \begin{bmatrix} \bS_{\cY} (I - T_{S} T_{S}^{*})^{\frac{1}{2}}  & P_{0} 
\end{bmatrix}  \begin{bmatrix} (I - T_{S} T_{S}^{*})^{\frac{1}{2}} S^{*} \\ 
P_{0} \end{bmatrix} \\
& \quad =   \bS_{\cY}(I - T_{S}T_{S}^{*}) 
\bS_{\cY}^{*} + I - \bS_{\cY} \bS_{\cY}^{*} = I - T_{S} T_{S}^{*}.
\end{align*}
Then by the Douglas criterion (Proposition \ref{P:Douglas}), there is 
a contraction operator $\left[ \begin{smallmatrix} X \\ Y 
\end{smallmatrix} \right]$ so that
$$
\begin{bmatrix} \bS_{\cY} (I - T_{S} T_{S}^{*})^{\frac{1}{2}} & P_{0} \end{bmatrix} 
    \begin{bmatrix} X \\ Y \end{bmatrix} = (I - T_{S} 
	T_{S}^{*})^{\frac{1}{2}}.
$$
Multiplying on the left by $\bS_{\cY}^{*}$ and then by $P_{0}$ 
successively breaks this up into the pair of equations
\begin{equation}
 (I - T_{S}T_{S}^{*})^{\frac{1}{2}} X = \bS_{\cY}^{ *} (I - T_{S}
    T_{S}^{*})^{\frac{1}{2}}, \qquad
    P_{0} Y   = P_{0} (I - T_{S} T_{S}^{*})^{\frac{1}{2}}.
  \label{pair}
\end{equation}
The first equation in \eqref{pair} reveals that $\cH(S)$ is 
invariant under $\bS_{\cY}^{*}$ and 
$$
\|\bS_{\cY}^{*} f\|^{2}_{\cH(S)}= \| X f_{1}\|^{2}_{H^{2}(\cY)}
$$
(assuming that it is arranged that $\operatorname{Ran} X \subset 
\overline{\operatorname{Ran}} (I - T_{S}T_{S}^{*})$ which is always 
possible).  Moreover,
the fact that $\left[ \begin{smallmatrix} X \\ Y \end{smallmatrix} 
\right]$ is a contraction implies that $Y$ has the form $Y = K (I - 
X^{*} X)^{\frac{1}{2}}$ with $K$ a contraction, and hence
$$
  P_{0}K(I - X^{*} X)^{\frac{1}{2}}  = P_{0} (I - T_{S} T_{S}^{*})^{\frac{1}{2}}.
$$
Then from the second equation in \eqref{pair} one gets
\begin{align*}
    \| f(0)\|^{2}_{\cY} & = \| P_{0} (I - T_{S} T_{S}^{*})^{\frac{1}{2}} 
    f_{1} \|^{2} _{\cY} \\
    & \le \| (I - X^{*} X)^{\frac{1}{2}} f_{1} \|^{2}_{H^{2}(\cY)} \\
    & = \| f_{1} \|^{2}_{H^{2}(\cY)} - \| X f_{1} \|^{2}_{H^{2}(\cY)} 
    = \| f \|^{2}_{\cH(S)} - \|\bS_{\cY}^{*} f \|^{2}_{\cH(S)},
\end{align*}
and the norm estimate \eqref{normest} follows.
 \end{proof}
 
 \begin{proof}[Proof of (3).]  Note that the subspace $\left\{ 
     \frac{S(z) - S(0)}{z} u \colon u \in \cU\right\}$ is the range 
     of the commutator operator $\bS_{\cY}^{*} T_{S} - T_{S} 
     \bS_{\cY}^{*}$.  Hence to show that $\frac{S(z) - S(0)}{z} u \in 
     \cH(S)$ for each $u \in \cU$, it suffices to show that 
     $\operatorname{Ran} (I - T_{S} T_{S}^{*})^{\frac{1}{2}}$ is invariant 
     under the commutator $\bS_{\cY}^{*} T_{S} - T_{S} 
     \bS_{\cY}^{*}$.  Again by Proposition 
     \ref{P:Douglas}, it suffices to show that 
 \begin{equation}   \label{Douglasineq}
 ( \bS_{\cY}^{*} T_{S} - T_{S} \bS_{\cY}^{*}) (\bS_{\cY}^{*} T_{S} - 
 T_{S}\bS_{\cY}^{*})^{*} \le I - T_{S} T_{S}^{*}.
 \end{equation}
It is readily seen that the left hand side of 
 \eqref{Douglasineq} is equal to $\bS_{\cY}^{*} T_{S}^{*} T_{S}^{*} 
 \bS_{\cY} - T_{S} T_{S}^{*}$, so \eqref{Douglasineq} does hold.

\smallskip
 
 Finally, the formula for $R_{0}^{*}$ can be verified as follows.  Assume 
 first that $h \in \cH(S)$ has the special form $h = (I - T_{S} T_{S}^{*}) 
 h_{1}$ for some $h_{1} \in H^{2}(\cY)$.  Then the computation
 \begin{align*}
     \langle R_{0} g, h \rangle_{\cH(S)} & = \langle \bS_{\cY}^{*} g, 
     h \rangle_{\cH(S)}   = \langle \bS_{\cY}^{*} g, h_{1} \rangle_{H^{2}(\cY)} 
      = \langle g, \bS_{\cY} h_{1} \rangle_{H^{2}(\cY)} \\ 
    &   = \langle g, (I - T_{S} T_{S}^{*}) \bS_{\cY} h_{1} 
     \rangle_{\cH(S)}
 \end{align*}
shows that 
 \begin{equation}  \label{RS*2}
 R_{0}^{*} h = (I - T_{S} T_{S}^{*}) \bS_{\cY} h_{1} = \bS_{\cY} h - 
 T_{S}(T_{S}^{*} \bS_{\cY} - \bS_{\cY} T_{S}^{*}) h_{1}.
 \end{equation}
 It is easily verified that
 $$
\bS_{\cY}^{*} T_{S} - T_{S} \bS_{\cY}^{*}  \colon h_{1}(z) \to 
 \frac{S(z) - S(0)}{z} h_{1}(0) = \tau ( h_{1}(0))
 $$
 and hence the adjoint action must have the form
 $$
 (T_{S}^{*} \bS_{\cY} - \bS_{\cY} T_{S}^{*}) h_{1} = \widetilde h_{1}(0)
 $$
 where the constant $\widetilde h_{1}(0) \in \cU$ is determined by
 \begin{align*}
 \langle \widetilde h_{1}(0), u \rangle_{\cU} & = \langle h_{1}, 
 \tau(u) \rangle_{H^{2}(\cY)}  = \langle (I - T_{S} T_{S}^{*}) h_{1}, 
 \tau(u) \rangle_{\cH(S)} = \langle h, \tau(u) \rangle_{\cH(S)} \\
 & = \langle \tau^{*}(h), u \rangle_{\cU}
 \end{align*}
 where the adjoint is with respect to the $\cH(S)$ inner product on 
 the range of $\tau$.  Therefore, $\widetilde h_{1}(0) = 
 \tau^{*}(h)$ and the formula \eqref{RS*1} for $R_{0}^{*}$ is now an 
 immediate consequence of \eqref{RS*2} for the case where $f$ has the 
 special form $f = (I - T_{S} T_{S}^{*}) f_{1}$.  But elements in 
 $\cH(S)$ of this special form are dense in $\cH(S)$ so the general 
 case of \eqref{RS*1} now follows by taking limits.
 
\smallskip

 The next task is the computation of the action of $I - R_{0} R_{0}^{*}$ on a general 
 element $f$ of $\cH(S)$:
 \begin{align*}
 (I - R_{0} R_{0}^{*})f &=f(z) - \frac{1}{z} [ z 
     f(z) - S(z) \tau^{*}(f) + S(0) \tau^{*}(f) ] \\
     & = \frac{S(z) - S(0)}{z} \tau^{*}(f) = \tau \tau^{*}(f).
 \end{align*}
Therefore,
 $$
 \| f \|^{2}_{\cH(S)} - \| R_{0}^{*} f \|^{2}_{\cH(S)}  =
 \langle (I - R_{0} R_{0}^{*}) f, f \rangle_{\cH(S)} 
 = \langle \tau \tau^{*}(f), f \rangle_{\cH(S)} = \| \tau^{*}(f) 
 \|^{2}_{\cU}
$$  
and the identity \eqref{normtau*} follows.
 \end{proof}
 
 The next goal is to show that the estimate \eqref{normest} is enough to verify 
 one direction of part (4) in Theorem \ref{T:H(S)}.
 
 \begin{theorem}   \label{T:cni}
     Let $S$ be in $\cS(\cU, \cY)$ with 
     associated de Branges-Rovnyak space $\cH(S)$ and model operator 
     $R_{0}: = \bS_{\cY}^{*}|_{\cH(S)}$.  Then $R_{0}$ is completely 
     non-isometric, i.e., if $f \in \cH(S)$ is such that
     $\| R_{0}^{n} f \|_{\cH(S)} = \| f \|_{\cH(S)}$ for 
     $n=0,1,2,\dots$, then $f = 0$.
 \end{theorem}
 
 \begin{proof}
     Suppose that $f(z) = \sum_{n=0}^{\infty} f_{n} z^{n}$ is in 
     $\cH(S)$ with $\|R_{0}^{n} f \| = \| 
     f \|$ for all $n=1,2,3, \dots$.  Then in particular, from the 
     observation that $f_{n} = (R_{0}^{n} f)(0)$ together with 
     \eqref{normest} one gets
 $$
   \| f_{n} \|^{2}  = \| ( R_{0}^{n}f)(0) \|^{2} _{\cY} \le 
   \| R_{0}^{n} f \|^{2}_{\cH(S)} - \| R_{0}^{n+1} f \|^{2}_{\cH(S)} 
   = 0
 $$
 and hence $f_{n} = 0$ for all $n=0,1,2, \dots$, i.e., $f = 0$.
 \end{proof}
 
 \subsection{Reproducing kernel Hilbert spaces}  \label{S:RKHS}
 A {\em reproducing kernel Hilbert space} (RKHS) by 
 definition is a Hilbert space whose elements are functions on some set $\Omega$ 
 with values in a coefficient Hilbert space, say $\cY$, such that the 
 evaluation map $\be(\omega) \colon f \mapsto f(\omega)$ is 
 continuous from $\cH$ into $\cY$ for each $\omega \in \Omega$.
 Associated with any such space is a positive $\cL(\cY)$-valued 
 kernel on $\Omega$, i.e., a function $K \colon \Omega \times \Omega 
 \to \cL(\cY)$ with the positive-kernel property
 \begin{equation}  \label{posker}
 \sum_{i,j=1}^{N} \langle K(\omega_{i}, \omega_{j}) y_{j}, y_{i} 
 \rangle_{\cY} \ge 0 
 \end{equation}
for any choice of finitely many points $\omega_{1}, \dots, \omega_{N}
 \in \Omega$ and vectors $y_{1}, \dots, y_{N} \in \cY$,
which ``reproduces'' the values of the functions in $\cH$ in the sense 
 that
 \begin{enumerate}
     \item[(i)] the function $\omega \mapsto K(\omega, \zeta) y$ 
     is in $\cH$ for each $\zeta \in \Omega$ and $y \in \cY$, and
     \item[(ii)] the reproducing formula 
     $$
       \langle f, K(\cdot, \zeta) y \rangle_{\cH} = \langle 
       f(\zeta), y \rangle_{\cY}
       $$
  holds for all $f \in \cH$, $\zeta \in \Omega$, and $y \in \cY$.
  \end{enumerate}
 Such kernels are also characterized by the Kolmogorov factorization 
 property:  {\em there exists a function $H \colon \Omega \to 
 \cL(\widetilde \cH, \cY)$ for some auxiliary Hilbert spaces 
 $\widetilde \cH$ so that }
\begin{equation}   \label{Kol}
  K(\omega, \zeta) = H(\omega) H(\zeta)^{*}.
\end{equation}
 One particular way to produce this factorization is by taking $\widetilde \cH = 
 \cH$ and setting  $H(\omega) = \be(\omega)$ where $\be(\omega)$ is 
 the point-evaluation map described above.  Whenever a Hilbert space 
 of functions arises in this way from a positive kernel $K$, one 
 writes $\cH = \cH(K)$. An early thorough treatment of RKHSs (for 
 the case $\cY = {\mathbb C}$) is the paper of Aronszajn 
 \cite{aron}; a good recent treatment is in the book of 
 Agler-McCarthy \cite{AgMcC} (where they are called {\em Hilbert 
 function spaces}) while the recent papers \cite{BV-formal} 
 formulate more general settings (formal commuting or noncommuting 
 variables).
 
\smallskip

 Given a pair of reproducing kernel Hilbert spaces $\cH(K_{0})$ and 
 $\cH(K)$ where say $\cH(K_{0})$ consists of functions with values in 
 $\cY_{0}$ and $\cH(K)$ consists of functions with values in $\cY$, 
 an object of much interest for operator theorists is the space of 
 multipliers $\cM(K_{0}, K)$ consisting of $\cL(\cY_{0}, \cY)$-valued 
 functions $F$ on $\Omega$ with the property that the multiplication 
 operator
 $$
    M_{F} \colon f(\omega) \mapsto F(\omega) f(\omega)
 $$
 maps $\cH(K_{0})$ into $\cH(K)$.  The simple computation 
 \begin{align*}
 \langle M_{F} f, K(\cdot, \zeta) y \rangle_{\cH(K)} &= \langle F(\zeta) 
 f(\zeta), y \rangle_{\cY}\notag\\ &= \langle f(\zeta), F(\zeta)^{*} y 
 \rangle_{\cY_{0}}= \langle f, K_{0}(\cdot, \zeta) F(\zeta)^{*} y 
 \rangle_{\cH(K_{0})}
 \end{align*}
 shows that
 \begin{equation}   \label{ShieldsWallen}
    (M_{F})^{*} \colon K(\cdot, \zeta) y \mapsto K_{0}(\cdot, \zeta) 
    F(\zeta)^{*} y.
 \end{equation}
Therefore
$$
\langle (I-M_FM_F^*)K(\cdot,\zeta)y, \, K(\cdot,\omega)y^\prime\rangle_{\cH(K)}=
\langle(K(\omega,\zeta)-F(z)K_0(\omega,\zeta)F(\zeta)^*)y, \, y^\prime\rangle_{\cY}
$$ 
so that $F$ is a contractive multiplier from $\cH(K_0)$ to $\cH(K)$ if and only if 
the kernel  $K(\omega,\zeta)-F(z)K_0(\omega,\zeta)F(\zeta)^*$ is positive on 
$\Omega\times\Omega$.
By letting $K_0(\omega,\zeta)\equiv I_\cY$ and rescaling, one can arrive at the following 
proposition 
\cite{beabur}.

\begin{proposition}
A function $F: \, \Omega\to \cY$ belongs to $\cH(K)$ with $\|F\|_{\cH(K)}\le \gamma$ if 
and only if 
the kernel $K(\omega,\zeta)-\gamma^{-2}F(z)F(\zeta)^*$ is positive on 
$\Omega\times\Omega$.
\label{bourb}
\end{proposition}

 A first example of a reproducing kernel Hilbert space is the 
 space $H^{2}(\cY)$ (considered as consisting of analytic functions 
 on the unit disk ${\mathbb D}$) with the Szeg\H{o} kernel tensored 
 with the identity operator on $\cY$:  $k_{\rm Sz}(z,w) I_{\cY}$ where 
 $k_{\rm Sz}(z,w) = \frac{1}{1 - z \overline{w}}$.  The space of 
 multipliers $\cM(k_{\rm Sz} I_{\cU}, k_{\rm Sz} I_{\cY})$ between 
 two Hardy spaces can be identified with the space 
 $H^{\infty}(\cU, \cY)$ of bounded analytic functions on 
 ${\mathbb D}$ with values in $\cL(\cU, \cY)$.  Given $F \in 
 H^{\infty}(\cU, \cY)$, the associated multiplication operator 
 $M_{F}$ is simply the Toeplitz operator $T_{F}$ which was discussed 
 above.  Note that in general $F^{*}$ is not a multiplier when $F$ is 
 a multiplier; however it does hold that $M_{F}^{*} = (T_{F})^{*} = 
 T_{F^{*}}$ for $F$ a multiplier between two Szeg\H{o}-kernel RKHSs 
 (i.e., Hardy spaces).
 
\smallskip

  The next task is the identification of 
 the de Branges-Rovnyak space $\cH(S)$ (where $S$ is a Schur-class 
 function) as a reproducing kernel Hilbert space in its own right.  
 This fills in part (1) of Theorem \ref{T:H(S)}.
 
 \begin{theorem} \label{T:H(S)1} The de Branges-Rovnyak space $\cH(S) = 
     \cH^{p}_{T_{S}}$ associated with a Schur-class function  
     $S\in\cS(\cU, \cY)$ as above is isometrically equal to the 
     reproducing kernel Hilbert space $\cH(K_{S})$ where $K_{S}$ is 
     the de Branges-Rovnyak kernel
  \begin{equation}  \label{deBRkernel}
      K_{S}(z,w) = \frac{I - S(z) S(w)^{*}}{1 - z \overline{w}}.
  \end{equation}
  \end{theorem}
 
  \begin{proof}
      As a result of the general identity \eqref{ShieldsWallen}, it follows that  
 \begin{equation}  \label{ts*}
        T_{S}^{*} \colon k_{\rm Sz}(\cdot, w) y \mapsto k_{\rm 
	Sz}(\cdot, w) S(w)^{*} y
 \end{equation}
  and hence 
  $$
    K_{S}(\cdot, w)y = (I - T_{S} T_{S}^{*}) (k_{\rm Sz}(\cdot ,w) y).
  $$
  It then follows that $K_{S}(\cdot, w) y \in \cH(S)$ for each $w \in 
  {\mathbb D}$ and $y \in \cY$, and also, for $f = (I - T_{S} T_{S}^{*}) f_{1} \in 
  \cH(S)$
  \begin{align*} 
  \langle f, K_{S}(\cdot, w) y \rangle_{\cH(S)}&  = \langle f, (I - 
  T_{S} T_{S}^{*}) (k_{\rm Sz}(\cdot, w) y \rangle_{\cH(S)} \\
  & = \langle f, k_{\rm Sz}(\cdot, w) y \rangle_{H^{2}(\cY)}= \langle f(w), y 
\rangle_{\cY}
  \end{align*}
  from which one can see that $K_{S}(z,w)$ has all the properties required 
  to be the reproducing kernel for $\cH(S)$.
 \end{proof}
 
  If $\be(w) \colon \cH(S) \to \cY$ is the evaluation-at-$w$ map
  $$
    \be(w) \colon f \mapsto f(w),
  $$
  on the space $\cH(S)$, then its adjoint is given by the kernel function for the point 
$w$:
  $$
    \be(w)^{*} \colon y \mapsto K_{S}(z, w) y = \frac{ I 
    - S(z) S(w)^{*}}{ 1 - z \overline{w}}.
  $$
  In particular,
  \begin{equation}   \label{e(0)*}
      \be(0)^{*} \colon y \mapsto (I - S(z) S(0)^{*}) y.
  \end{equation}
  This is the last piece needed to complete the proof of part (4) of 
  Theorem \ref{T:H(S)}.  Along the way, here also is a completion of the analysis 
  of when the inequality \eqref{norm-est} holds with equality.
  
  \begin{theorem}  \label{T:H(S)4}
      The colligation matrix $\bU_{S}=\left[
      \begin{smallmatrix} A_S & B_S \\ C_S & D_S\end{smallmatrix}\right] \colon \left[ 
      \begin{smallmatrix} \cH(S) \\ \cU \end{smallmatrix} \right]\to 
	\left[  \begin{smallmatrix} \cH(S) \\ \cY \end{smallmatrix} \right]$ given by 
\eqref{bUS} is 
      coisometric and has characteristic function equal to $S$:
 \begin{equation}   \label{realization'}
     S(z) = D_{S} + z C_{S}(I - zA_{S})^{-1} B_{S}.
 \end{equation}
 Furthermore, the reproducing kernel $K_{S}(z,w)$ can be expressed 
 directly in terms of the colligation matrix $\bU_{S}$:
 \begin{equation}   \label{ker-bU}
     \frac{I - S(z) S(w)^{*}}{1 - z \overline{w}} = C_{S} (I - zA_{S})^{-1} 
     (I - \overline{w} A_{S}^{*})^{-1} C_{S}^{*}.
  \end{equation}
      Moreover:
      \begin{enumerate}
	 \item  The following are equivalent:
	  \begin{enumerate}
	  \item $\bU_{S}$ is unitary.  
     \item  $S$ satisfies the  condition
      \begin{equation} \label{bUunitary}
	  S(z) u \in \cH(S) \Rightarrow u = 0.
\end{equation}
\item The maximal factorable minorant of $I - S^{*}S$ is zero, i.e.,
if $\Phi \in \cS(\cU, \cY_{0})$ satisfies $\Phi(\zeta)^{*} 
\Phi(\zeta) \le I - S(\zeta)^{*} S(\zeta)$ for almost all $\zeta \in 
{\mathbb T}$, then $\Phi = 0$.
\end{enumerate}

   \item  The following are equivalent:
   \begin{enumerate}
       \item The difference quotient identity holds, i.e.,
equality holds in \eqref{norm-est}.
\item $S$ satisfies the condition
\begin{equation}  \label{diffquoteq}
    S(z) u \in \cH(S) \Rightarrow S(z) u \equiv 0.
\end{equation}
\item If $S'$ is the normalization of $S$ as in \eqref{normalize} (so 
$S'$ satisfies \eqref{normalize} and $\cH(S) = \cH(S')$),
then the maximal factorable minorant of $S^{\prime *}S' $ is zero. 
\end{enumerate}
\end{enumerate}
\end{theorem}
  
  \begin{proof}
      The coisometry property of $\bU_{S}$ amounts to the three 
      identities
 \begin{align}  
   R_{0} R_{0}^{*} + \tau \tau^{*} & = I_{\cH(S)}, \label{id1} \\
   \be(0) \be(0)^{*} + S(0) S(0)^{*} & = I_{\cY}, \label{id2} \\
   \be(0) R_{0}^{*} + S(0) \tau^{*} & = 0. \label{id3}
 \end{align}
 The identity \eqref{id1} is the same as \eqref{RS*norm} which has been 
 already verified above.  From the formula \eqref{e(0)*} it is 
 immediate that
 $$
  \be(0) \be(0)^{*} \colon y \mapsto \be(0) \left(I - S(z) 
  S(0)^{*} \right)y =(I - S(0) S(0)^{*})y
 $$
 and \eqref{id2} now follows.  Next, use the formula \eqref{RS*} for 
 $R_{0}^{*}$ to compute
 $$
 \be(0) R_{0}^{*} \colon f(z) \mapsto - S(0) \tau^{*}(f)
 $$ from which \eqref{id3} is now immediate.  This completes the 
 verification of the coisometry property of $\bU_{S}$.

\smallskip
 
 To verify \eqref{realization}, let $S(z) = \sum_{n=0}^{\infty} S_{n} 
 z^{n}$ be the Taylor series for $S(z)$. Then the computation
 \begin{align*}
     \left( D_{S} + w C_{S}(I - w A_{S})^{-1} B_{S} \right) u & =
     S(0) u + w \sum_{n=1}^{\infty} w^{n} \be(0) \bS_{\cY}^{*n} \left( 
     \sum_{k=1}^{\infty} S_{k+1} z^{k} \right) u \\
     & = S(0) u + w \sum_{n=1}^{\infty} w^{n} S_{n+1} u = S(w) u
 \end{align*}
 verifies the realization formula \eqref{realization'}.
 The formula \eqref{ker-bU} can be verified by direct computation: 
 plug in the formula \eqref{realization'} for $S(z)$ and use that 
 $\bU_{S}$ is coisometric.
 
\smallskip

 Once it is known that $\bU_{S}$ is coisometric, it follows that 
 $\bU_{S}$ is unitary if and only if $\operatorname{Ker} \bU_{S} = 
 \{0\}$.  Note that $\left[ 
 \begin{smallmatrix} f(z) \\ u \end{smallmatrix} \right]$ being in 
     $\operatorname{Ker} \bU_{S}$ means that 
$$ \begin{bmatrix} 0 \\ 0 \end{bmatrix} = \bU_{S} \begin{bmatrix} f(z) \\ u \end{bmatrix} 
=
\begin{bmatrix} [f(z) - f(0)]/z + [S(z) u - S(0) u]/z \\  f(0) + S(0) u 
    \end{bmatrix}.
    $$
Thus
$$
 f(z) = f(0) + z \cdot \frac{ f(z) - f(0)}{z} 
 = -S(0)u - z \cdot \frac{S(z) - S(0)}{z} u  = -S(z) u.
$$
This completes the proof of the equivalence of (a) and (b) in part  (1) in the theorem.
   
\smallskip

 A computation of Nikolskii-Vasyunin (see  \cite[Theorem 
    8.7]{NV0}) gives the following:
   for $f \in \cH(S)$ of the special form $f = (I - T_{S} 
    T_{S}^{*}) f_{1}$,
    \begin{align}  
  &  \| f \|^{2}_{\cH(S)} - \| R_{0} f \|^{2}_{\cH(S)} \notag \\
  & \quad =
    \| f(0) \|^{2} + \inf\{ \| (I - T_{S}^{*} T_{S})^{\frac{1}{2}} 
    (T_{S}^{*}f_{1})(0) + \bS_{\cY} g \|^{2}_{H^{2}(\cU} \colon g \in 
    H^{2}(\cU)\}.\notag
    \end{align}
Therefore the operator $\left[ \begin{smallmatrix} A_{S} \\ C_{S} 
\end{smallmatrix} \right] = \left[ \begin{smallmatrix} R_{0} \\ 
\be(0) \end{smallmatrix} \right]$ is isometric if and only if
\begin{equation}  \label{NVconclusion1}
\inf\{ \| (I - T_{S}^{*} T_{S})^{\frac{1}{2}} 
    (T_{S}^{*}f_{1})(0) + \bS_{\cY} g \|^{2}_{H^{2}(\cU} \colon g \in 
    H^{2}(\cU)\} = 0 \quad\text{for all}\quad g \in H^{2}(\cU).
 \end{equation}
 Another computation of Nikolskii-Vasyunin (see \cite[Lemma 
 i.2]{NV0}) gives:
\begin{equation}  \label{NVid2}
    \| \tau(u) \|^{2}_{\cH(S)} = \| u \|^{2}_{\cU} - \| S(0) u 
    \|^{2}_{\cY} - \inf \{ \| (I - T_{S}^{*} T_{S})^{\frac{1}{2}}(u + 
    \bS_{\cU} 
    g) \|^{2} \colon g \in H^{2}(\cU) \}.
\end{equation}
Therefore the operator $\left[ \begin{smallmatrix} B_{S} \\ D_{S} 
\end{smallmatrix} \right] = \left[ \begin{smallmatrix} \tau \\ 
S(0) \end{smallmatrix} \right]$ is isometric  if and only if
\begin{equation}   \label{NVconclusion2}
     \inf \{ \| (I - T_{S}^{*} T_{S})^{\frac{1}{2}}(u + \bS_{\cU} 
    g) \|^{2} \colon g \in H^{2}(\cU) \} = 0 \text{ for all } u \in 
    \cU.
\end{equation}
Note that condition \eqref{NVconclusion2} implies 
\eqref{NVconclusion1}.  The condition \eqref{NVconclusion2} amounts 
to the statement that  vector analytic polynomials in $z$ are dense in the 
    weighted space $L^{2}(\cU)$ with $(I - S(\zeta)^{*} S(\zeta)) 
    |d\zeta|$-metric.  It is well known how this condition in turn 
    translates to $0$ is the maximal factorable minorant for $I - 
    S(\zeta)^{*} S(\zeta)$ (see e.g.~ \cite[Proposition V.4.2]{NF}).
    In this way one can see that the zero maximal-factorable-minorant 
    condition is equivalent to each column $\left[ 
    \begin{smallmatrix} A_{S} \\ C_{S} \end{smallmatrix} \right]$ and
	 $\left[ 
    \begin{smallmatrix} B_{S} \\ D_{S} \end{smallmatrix} \right]$
	of the colligation matrix $\bU_{S}$ being isometric.
 As the isometry property of $U_{S}^{*}$ has already been verified 
 above, it follows that $\bU_{S}$ is 
contractive.  The next elementary exercise is to verify in general 
that a contractive $2 \times 2$ block operator matrix $\bU_{S} =
\left[ \begin{smallmatrix} A_{S} & B_{S} \\ C_{S} & D_{S} 
\end{smallmatrix} \right]$ with each column isometric must itself be 
isometric.  In this way the equivalence of (a) and (c) 
in statement (1) of the theorem follows, and (1) follows as well.

\smallskip
    
    It remains to verify the equivalence of (a) and (c) with the 
     normalization assumption \eqref{normalize} imposed. For 
     simplicity  $S$ rather than $S'$ is written with the assumption that $S$ 
     satisfies \eqref{normalize}.  Then one can show  
    that the set of elements in $\cU$ of the form 
    $(T_{S}^{*}f_{1})(0)$with $f_{1} \in H^{2}(\cY)$ is dense in $\cU$. Then a limiting 
    argument implies that equality 
    holding in \eqref{norm-est} for all $f \in \cH(S)$ is equivalent 
    to the condition \eqref{NVconclusion2}.  As explained in the 
    previous paragraph, \eqref{NVconclusion2} is equivalent to the 
    zero maximal-factorable-minorant condition.
    The reverse implication follows by reversing the argument.
\end{proof}
    
 \begin{remark} {\em     The equivalence of (a) and (c) in part (2) 
	of Theorem \ref{T:H(S)4} was already observed 
    in \cite[Theorem 6]{BK}.  The idea of the proof is as follows.
    If $\Phi^{*} \Phi$ is the maximal factorable minorant for $I - 
    S^{*} S$ (with $\Phi$ outer), then there is a unitary 
    transformation $J \colon \cH(S) \to {\mathbb K}\left( \left[ 
    \begin{smallmatrix} S \\ \Phi \end{smallmatrix} \right] \right)$,
where ${\mathbb K}\left( \left[\begin{smallmatrix} S \\ \Phi \end{smallmatrix} 
\right] \right)$ is the Sz.-Nagy-Foias model space based on the Schur-class 
function $\left[ \begin{smallmatrix} S \\ \Phi \end{smallmatrix} \right]$
which intertwines the de Branges-Rovnyak model operator $R_{0}$ with 
the adjoint of the Sz.-Nagy-Foias model operator $( T_{\left[ 
\begin{smallmatrix} S \\ \Phi \end{smallmatrix} \right]} )^{*}$; on 
    kernel functions the map $J$ has the form
 $$
  J \colon K_{S}(\cdot , w) y \mapsto \begin{bmatrix} K_{S}(\cdot , w) y \\ - 
  \Phi S(w)^{*} k_{\rm Sz}(\cdot, w) y \\ - \nabla S(w)^{*} k_{\rm 
  Sz}(\cdot w) y \end{bmatrix}
 $$
 where 
 $$
   \nabla(\zeta) = (I - S(\zeta)^{*} S(\zeta) - \Phi(\zeta)^{*} 
   \Phi(\zeta) )^{\frac{1}{2}}  \text{ for } \zeta \in {\mathbb T}.
 $$
 Note that the first two components of elements of ${\mathbb K}\left( 
 \left[ \begin{smallmatrix} S \\ \Phi \end{smallmatrix} \right] 
 \right)$ are analytic functions functions on ${\mathbb D}$ while the third component is
 a measurable function on ${\mathbb T}$, and that $J$ has the form 
 $$
    J \colon f \mapsto \begin{bmatrix} f \\ g_{f} \\ w_{f} \end{bmatrix},
 $$
 i.e., the projection of $J$ to the first component is the identity map.  
 One next observes that
 $$\| R_{0} f \|^{2}_{\cH(S)} = \left \| ( T_{\left[ \begin{smallmatrix} S \\ \Phi
\end{smallmatrix} 
 \right]})^{*} \left[ \begin{smallmatrix} f \\ g_{f} \\ w_{f} 
\end{smallmatrix} \right] \right\|^{2} = \| f \|^{2}_{\cH(S)} - \|f(0)\|^{2} - 
 \|g_{f}(0)\|^{2}.
$$
Hence the difference quotient identity (equality in \eqref{norm-est}) 
holding for all $f \in \cH(S)$ is equivalent to $g_{f}(0) = 0$ for 
all $f \in \cH(S)$.  Under the assumption \eqref{normalize}, it can be shown 
that this is equivalent to the normalized version $S'$ of $S$ (i.e., $S'$ 
satisfying \eqref{normalize} while $\cH(S) = \cH(S')$) having maximal 
factorable minorant equal to $0$.\footnote{This last point was 
missed in \cite{BK}; the normalization condition  \eqref{normalize} 
was not mentioned explicitly.}

\smallskip

The equivalence of (a) and (c) in part (1) of Theorem \ref{T:H(S)4} 
also follows from results of \cite{BK}.   For this alternative proof, 
the following fact will be used:  {\em the colligation matrix 
$\bU_{S}$ being unitary is equivalent to the isometric embedding of 
$\cH(S)$ into the two-component space $\cD(S)$ being onto.}  Theorem 8 
of \cite{BK} shows that this happens if and only if the maximal 
factorable minorant of $I - S^{*} S$ is zero.
}\end{remark}

To this point the operator-range characterization of 
$\cH(S)$  has been used to develop the basic properties of the operators $R_{0}, 
\tau, \be(0), S(0)$ in the colligation matrix $\bU_{S}$. 
Alternatively, the space $\cH(S)$ could have been defined as the RKHS with 
reproducing kernel $K_{S}$ and this characterization then used to obtain 
the results concerning $\bU_{S}$.  To see directly that $K_{S}$ is a 
positive kernel (without recourse to the operator-range characterization 
of $\cH(S)$), it suffices to note that the Toeplitz operator $T_{S}$ 
has $\| T_{S} \| \le 1$ for $S \in \cS(\cU, \cY)$ (since the 
boundary-value function $\zeta \mapsto S(\zeta)$ on ${\mathbb T}$ has 
contractive values) and hence
\begin{equation} \label{contrac}
    \| f \|^{2}_{H^2(\cY)} - \| T_{S}^{*} f \|^{2}_{H^2(\cU)} \ge 0 
    \quad\text{for all}\quad f \in H^{2}(\cY).
\end{equation}
Set $f = \sum_{j=1}^{N} k_{\rm Sz}(\cdot, w_{j})y_{j} \in 
H^{2}(\cY)$. Then condition \eqref{contrac} translates to 
\eqref{posker}, and it follows that $K_{S}$ is a positive kernel and 
hence one can define $\cH(S)$ as the reproducing kernel Hilbert space 
$\cH(K_{S})$.   The following discussion presents an alternate proof 
of parts (2), (3), and (4) of Theorem \ref{T:H(S)} using the 
reproducing-kernel-space characterization (i.e., part (1) of Theorem 
\ref{T:H(S)})  rather than the operator-range characterization 
\eqref{def2HS} of the space $\cH(S)$.

\begin{proof}[Proof of parts (2), (3), (4) of Theorem \ref{T:H(S)} 
    based on part (1)]  Given that $S$ is in the Schur class $\cS(\cU, \cY)$, 
it has been  explained in the previous paragraph why $K_{S}$ is a positive 
    kernel (i.e., satisfies \eqref{posker}) and hence generates a 
    well-defined RKHS $\cH(K_{S})$.  By the general theory of RKHSs 
    (see \eqref{Kol} above and the explanation there),
    it is known that $K_{S}$ has its canonical Kolmogorov decomposition
 \begin{equation}   \label{Kolfact'}
     K_{S}(z,w) = \be(z) \be(w)^{*}
\end{equation}
where $\be(z) \colon \cH(K_S) \to \cY$ is the point-evaluation map 
$\be(z) \colon f \mapsto f(z)$ for each $z \in {\mathbb D}$.  
Substituting $K_{S}(z,w) = [I - S(z) S(w)^{*}]/(1 - z \overline{w})$ 
and rearranging \eqref{Kolfact'} leads to
$$
z \overline{w} \be(z) \be(w)^{*} + I_{\cY} = \be(z) \be(w)^{*} + S(z) 
S(w)^{*}.
$$
The inner-product identity
$$
\langle \overline{w} \be(w)^{*} y, \overline{z} \be(z)^{*} y' 
\rangle_{\cH(K_{S})} + \langle y, y' \rangle_{\cY} = 
\langle \be(w)^{*} y, \be(z)^{*} y' \rangle_{\cH(K)} + \langle 
S(w)^{*} y, S(z)^{*} y' \rangle_{\cU}
$$
then follows, where $y,y'$ are arbitrary vectors in $\cY$.
This inner-product identity can be written in aggregate form
$$
\left\langle \begin{bmatrix} \overline{w} \be(w)^{*} \\ I \end{bmatrix} 
y, \begin{bmatrix} \overline{z} \be(z)^{*} \\ I \end{bmatrix} y' 
\right\rangle_{\cH(K_S) \oplus \cY} = \left\langle \begin{bmatrix} \be(w)^{*} \\ S(w)^{*} 
\end{bmatrix} y, \begin{bmatrix} \be(z)^{*} \\ S(z)^{*} \end{bmatrix} 
y' \right \rangle_{\cH(K_S) \oplus \cU}.
$$
It follows that the mapping
\begin{equation}   \label{defV1}
    V \colon \begin{bmatrix} \overline{w} \be(w)^{*} \\ I 
\end{bmatrix} y  \mapsto \begin{bmatrix} \be(w)^{*} \\ S(w)^{*} 
\end{bmatrix} y
\end{equation}
extends by linearity and continuity to an isometry from
$$
\cD = \overline{\operatorname{span}} \left\{ \begin{bmatrix} 
\overline{w} \be(w)^{*} \\ I \end{bmatrix} y \colon w \in {\mathbb 
D}, y \in \cY \right\}
$$
onto
$$
\cR = \overline{\operatorname{span}} \left\{ \begin{bmatrix} 
\be(w)^{*} \\ S(w)^{*} \end{bmatrix} y \colon w \in {\mathbb D}, y 
\in \cY \right\}.
$$
Taking $w=0$ in the expression for a given element of $\cD$ reveals 
 that $\cD \supset \left[ \begin{smallmatrix} \{0\} \\ \cY 
\end{smallmatrix} \right]$.  Since the kernel elements $\{ \be(w)^{*} 
y  = K_{S}(\cdot, w) y \colon w \in {\mathbb D} \setminus \{0\},\, y 
\in \cY \}$ are dense in $\cH(K_{S})$, it next follows that in fact $\cD$ 
is the whole space $\cD = \left[ \begin{smallmatrix} \cH(K_{S}) \\ 
\cY \end{smallmatrix} \right]$ and $V$ is defined on the whole space
 $\cD = \left[ \begin{smallmatrix} \cH(K_{S}) \\ 
\cY \end{smallmatrix} \right]$.  Write out $V$ in block-matrix form
$$
  V = \begin{bmatrix} A^{*} & C^{*} \\ B^{*} & D^{*} \end{bmatrix} 
  \colon \begin{bmatrix} \cH(K_{S}) \\ \cY \end{bmatrix} \to 
  \begin{bmatrix} \cH(K_{S}) \\ \cU \end{bmatrix}.
$$
It then follows  from \eqref{defV1} that
$$
    \begin{bmatrix} A^{*} & C^{*} \\ B^{*} & D^{*} \end{bmatrix} 
	\colon \begin{bmatrix} \overline{w} \be(w)^{*} \\ I 
    \end{bmatrix} y  = \begin{bmatrix} \be(w)^{*} \\ S(w)^{*} 
\end{bmatrix} y
$$
or
\begin{equation}   \label{defV2}
    \left\{ \begin{array}{rcl}
    \overline{w} A^{*} \be(w)^{*}y + C^{*} y & = & \be(w)^{*} y \\
    \overline{w} B^{*} \be(w)^{*}y + D^{*} y & = & S(w)^{*} y.
    \end{array} \right.
\end{equation}
The first equation can be solved for $\be(w)^{*} y$  (note that 
$\|A\| \le 1$ since $V$ is isometric and hence the inverse $(I - 
\overline{w} A^{*})^{-1}$ is well-defined for all $w \in {\mathbb 
D}$):
$$
\be(w)^{*} y = (I - \overline{w} A^{*})^{-1} C^{*} y.
$$
The second equation then implies
$$
\left(\overline{w} B^{*} (I - \overline{w} A^{*})^{-1} C^{*}  + D^{*} 
\right) y = S(w)^{*} y.
$$
Cancelling off the vector $y$, taking adjoints, and replacing the variable $w$   
with the variable $z$ then gives
$$
  S(z) = D + z C (I - zA)^{-1} B \quad\text{where}\quad \begin{bmatrix} A & B 
  \\ C & D \end{bmatrix} = V^{*}\quad \text{is a coisometry.}
$$
Putting the pieces together leads to part (4) of Theorem \ref{T:H(S)} apart from 
making the identification $V^{*} = \bU_{S}$.

\smallskip

Letting $w=0$ in \eqref{defV2} enables one to solve for $C^{*}$:
$$
C^{*} y = \be(0)^{*} y =  K_{S}(\cdot, 0) y.
$$
The simple duality computation
$$
  \langle Cf, y \rangle_{\cY} = \langle f, C^{*} y 
  \rangle_{\cH(K_{S})} = \langle f, K_{S}(\cdot, 0) y 
  \rangle_{\cH(K_{S})} = \langle f(0), y \rangle_{\cY}
$$
shows that $C = \be(0)$.  One can use \eqref{defV2} 
to compute the action of $A^{*}$ on kernel elements $\be(w)^{*} y$ as 
follows:
$$
A^{*} \be(w)^{*} y = \frac{1}{\overline{w}} \left( \be(w)^{*} - 
\be(0)^{*} \right) y.
$$
Another duality computation
\begin{align*}
    \langle Af, \be(w)^{*} y \rangle_{\cH(K_{S})} & = \langle f, 
    A^{*} \be(w)^{*} y \rangle_{\cH(K_{S})}  \\
    & = \langle f, \frac{1}{\overline{w}} \left( \be(w)^{*} - 
    \be(0)^{*} \right) y \rangle_{\cH(K_{S})} \\
    & = \left\langle[f(w) - f(0)]/w, y \right\rangle_{\cY}
\end{align*}
leads to the conclusion that
$$
   A \colon f(z) \mapsto [f(z) - f(0)]/z, 
$$
i.e., that $A =  R_{0} = \bS_{\cY}^{*}|_{\cH(K_{S})}$.  Application of 
the second equation in \eqref{defV2} with $w=0$ yields
that $D^{*} = S(0)^{*}$, i.e., $D = S(0)$.  A return to the second 
equation in \eqref{defV2} 
then implies that
$$
  B^{*} \be(w)^{*} y \mapsto \left( [S(w)^{*} - 
  S(0)^{*}]/\overline{w} \right) y.
$$
Then the computation
\begin{align*}
    \langle B u, \be(w)^{*} y \rangle_{\cH(K_{S})}& = \langle u, [ 
   \left( S(w)^{*} - S(0)^{*}]/\overline{w} \right) y \rangle_{\cU} \\
   & = \langle \left( [ S(w) - S(0)]/w \right) u, y \rangle_{\cY}
\end{align*}
verifies that $B \colon u \mapsto \left( [S(z) - S(0)]/z 
\right) u$, i.e., $B = \tau$ where $\tau$ is as in \eqref{tau}.
It has been now verified that $\left[ \begin{smallmatrix} A & B \\ C & D 
\end{smallmatrix} \right] = \bU_{S}$.  Moreover \eqref{norm-est} and 
\eqref{RS*norm} are immediate consequences of the already derived 
property that $\bU_{S} = \left[ \begin{smallmatrix} A & B \\ C & D 
\end{smallmatrix} \right]$ is a coisometry.
    \end{proof}
    
    \begin{remark}  \label{R:lurking} {\em The above construction has
	become known as 
	the {\em lurking isometry argument} (see \cite{Ball-Winnipeg} 
	where this term was first coined).  In this single-variable 
	setting, it turns out that the isometry is defined on the 
	whole space; in other applications (see 
	\cite{Ball-Winnipeg}), the isometry is defined only on a 
	subspace and one must extend it to an isometry (or unitary or 
	contraction depending on what is wanted) on the whole space.
	}\end{remark}
    
The preparation is now laid for the use RKHS techniques to prove part (5) of Theorem 
\ref{T:H(S)}.  The following is an enhanced version  of  the 
necessity direction of part (5) of Theorem \ref{T:H(S)};  note 
that the  sufficiency direction is handled in Theorem \ref{T:cni} 
above.  The following proof is based on the RKHS characterization of 
$\cH(S)$ (part (1) of Theorem \ref{T:H(S)}).

\begin{theorem} \label{T:H(S)5}
    Let $A$ be a completely non-isometric contraction operator on a 
    Hilbert space $\cX$.   Then there is a Schur-class function $S 
    \in \cS(\cU, \cY)$ so that $T$ is unitarily equivalent to the 
    model operator $R_{0} = \bS_{\cY}^{*}|_{\cH(S)}$ on $\cH(S)$.  
    Furthermore one can arrange that $I - S^{*}S$ have maximal factorable 
    minorant equal to $0$.
    \end{theorem}
    
\begin{proof} 
    Let $A$ be any completely non-isometric contraction 
    operator on a Hilbert space $\cX$.  Choose an operator $C$ from 
    $\cX$ to a coefficient Hilbert space $\cY$ so that the block 
    column operator $\left[ \begin{smallmatrix} A \\ C 
\end{smallmatrix} \right] \colon \cX \to \left[ \begin{smallmatrix} 
\cX \\ \cY \end{smallmatrix} \right]$ is an isometry, i.e., so that
$$
   C^{*} C = I_{\cX} - A^{*} A.
$$
Note that one way to do this is to take $\cY = \cD_{A}$ equal to the 
{\em defect space} of $A$
$$
  \cD_{A} = \overline{\operatorname{Ran}} D_{A}
$$
where $D_{A}$ is the {\em defect operator} of $A$:
$$
  D_{A} = ( I - A^{*} A)^{\frac{1}{2}}.
$$
Consider the operator $\cO_{C,A} \colon \cX \to H^{2}(\cY)$ defined by
\begin{equation}
\cO_{C,A} \colon x \mapsto C (I - zA)^{-1} x.
\label{obs}
\end{equation}
The notation $\cO_{C,A}(z) = C (I - zA)^{-1}$ is useful; then
$\left(\cO_{C,A} x \right)(z) = \cO_{C,A}(z) x$.
To see that $\cO_{C,A}$ maps into $H^{2}(\cY)$,  note that 
$\cO_{C,A} x$ has Taylor series
$$
\cO_{C,A}(z) x = \sum_{n=0}^{\infty} (C A^{n} x) z^{n}.
$$
The computation of the $H^{2}(\cY)$-norm of $\cO_{C,A}x$ can be 
organized as follows:
\begin{align*}
    \| {\cO}_{C,A} x \|^{2}_{H^{2}(\cY)} & = 
    \sum_{n=0}^{\infty} \| C A^{n} x \|^{2}_{\cY} = \sum_{n=0}^{\infty} 
    \langle A^{*n} C^{*} C A^{n} x, x \rangle_{\cX}  \\
    & =
    \sum_{n=0}^{\infty} \langle A^{*n} (I - A^{*} A) A^{n} x, x 
    \rangle_{\cX} = \sum_{n=0}^{\infty} \left[ \| A^{n} x \|^{2}_{\cX} - \| 
    A^{n+1} x \|^{2}_{\cX} \right] \\
    & = \| x \|^{2}_{\cX} - \lim_{N \to \infty} \| A^{N} x \|^{2}_{\cX} < \infty
\end{align*}
(where the limit exists since the sequence $\{ \| A^{N} x \| \}_{N 
\ge 0}$ is nonincreasing).  For reasons connected with system theory 
which are not discussed here, the notation $\cO_{C,A}$ is sued to 
suggest the {\em observability operator} for the {\em output pair} $(C,A)$.
Note that $x \in \operatorname{Ker} \cO_{C,A}$
if and only if $C A^{n} x = 0$ for $n=0,1,2,\dots$, or equivalently, 
if and only if 
$$
0 = \langle A^{*n}C^{*}C A^{n} x, x \rangle_{\cX} = \langle A^{*}(I - 
A^{*} A) A^{n}x, x \rangle = \| A^{n}x\|^{2}_{\cX} - \| A^{n+1} x 
\|^{2}_{\cX}.
$$
Thus $\|A^{n} x \| = \|x \|$ for all $n=1,2, \dots$, or 
$\operatorname{Ker} \cO_{C,A}$ is the 
maximal invariant subspace for $A$ on which $A$ is isometric.  The 
assumption that $A$ is completely non-isometric implies that $x=0$, 
and hence the observability operator $\cO_{C,A}$ is 
one-to-one.  Define the Hilbert space $\cH_{0}$ as
$\cH_{0} = \operatorname{Ran} \cO$ with the lifted norm:
$$
   \| \cO_{C,A} x \|_{\cH_{0}} = \| x \|_{\cX},
$$
i.e., $\cH_{0}$ is the pullback space $\cH_{0} = \cH^{p}_{\cO_{C,A}}$.
The computation
$$ 
\bS_{\cY}^{*}\cO_{C,A} x = \sum_{n=0}^{\infty} (C A^{n+1} x) 
z^{n} = \left(\sum_{n=0}^{\infty} C A^{n} x z^{n} \right) A = \cO_{C,A} (A x)
$$
verifies the intertwining relation
$$
  \bS_{\cY}^{*}\cO_{C,A} = \cO_{C,A} A.
$$
As by definition $\cO_{C,A}$ is a unitary transformation from $\cX$ 
onto $\cH_{0}$, it follows that $A$ is unitarily equivalent to the 
operator $R_{0} \in \cL(\cH_{0})$ given by $R_{0} = 
\bS_{\cY}^{*}|_{\cH_{0}}$.

\smallskip

It remains to identify the space $\cH_{0}$ more explicitly  as a RKHS.  
Since $\cH_{0}$ sits inside 
$H^{2}(\cY)$, the point-evaluation maps $f \mapsto f(w)$ are 
well-defined for all $w \in {\mathbb D}$. Therefore
\begin{align*}
    \langle f(w), y \rangle_{\cY} & = \langle C (I - wA)^{-1} x, y 
    \rangle_{\cY} =  \langle x, (I - \overline{w} A^{*})^{-1} y 
    \rangle_{\cX} \\
    & = \langle \cO_{C,A} x, \cO_{C,A} 
    (I - \overline{w} A^{*})^{-1} y  \rangle_{\cH_{0}} \\
    & = \langle f, \cO_{C,A} (I - \overline{w} A^{*})^{-1} 
    C^{*} \rangle_{\cH_{0}}
\end{align*}
where 
$$
\left(\cO_{C,A} (I - \overline{w} A^{*})^{-1} C^{*} y
\right)(z) = C (I - zA)^{-1} (I - \overline{w} A^{*})^{-1} C^{*}y =: 
K_{C,A}(z,w) y.
$$
It now follows that $\cH_{0}$ is a RKHS with reproducing kernel equal to 
$K_{C,A}$.  The claim to be checked now is that $K_{C,A}$ in fact has the form $K_{S}(z,w) 
= [I - S(z) S(w)^{*}]/(1 - z \overline{w})$ for  a Schur-class 
function $S \in \cS(\cU, \cY)$ for an appropriate coefficient Hilbert 
space $\cU$.

\smallskip

Toward this end, a useful observation is that the formula \eqref{ker-bU} in Theorem 
\ref{T:H(S)4} is quite general:  {\em if $\bU = \left[ 
\begin{smallmatrix} A & B \\ C & D \end{smallmatrix} \right] \colon 
    \left[ \begin{smallmatrix} \cX \\ \cU \end{smallmatrix} \right] 
    \to \left[ \begin{smallmatrix} \cX \\ \cY \end{smallmatrix} 
    \right]$ is 
    coisometric. The function $S$ given by $S(z) = D + z C (I - zA)^{-1} B$
    is in the Schur class $\cS(\cU, \cY)$ and}
 $$
 \frac{I - S(z) S(w)^{*}}{1 - z \overline{w}} = C (I - zA)^{-1} (I - 
 \overline{w} A^{*})^{-1} C^{*} = K_{C,A}(z,w).
 $$
 Thus, identification of $K_{C,A}$ as having the form $K_{S}$ requires only 
 a solution of the matrix completion problem: {\em given the isometric output 
 pair $(C,A)$} (so $A^{*} A + C^{*} C = I$), {\em find a block-column 
 operator matrix $\left[ \begin{smallmatrix} B \\ D \end{smallmatrix} 
 \right]$ so that the operator matrix $\bU = \left[ 
 \begin{smallmatrix} A & B \\ C & D \end{smallmatrix} \right]$ is 
     coisometric.}  But this is easily done by solving a Cholesky 
     factorization problem:  {\em find $\left[ \begin{smallmatrix} B 
     \\ D \end{smallmatrix} \right]$ so that}
 $$
 \begin{bmatrix} B \\ D \end{bmatrix} \begin{bmatrix} B^{*} & D^{*} 
     \end{bmatrix} = \begin{bmatrix} I & 0 \\ 0 & I \end{bmatrix} 
     - \begin{bmatrix} A \\ C \end{bmatrix} 
     \begin{bmatrix} A^{*} & C^{*} \end{bmatrix}.
$$
Given that $\left[ \begin{smallmatrix} A \\ C \end{smallmatrix} 
\right]$ is isometric, it is not difficult to see that this problem 
is solvable; in fact, one can arrange that $\left[ 
\begin{smallmatrix} B \\ D \end{smallmatrix} \right]$ is injective 
    and then $\bU$ will be unitary (not just coisometric).  
    Furthermore, an adaptation of the lurking isometry argument 
    (see Remark \ref{R:lurking}) in 
    the proof of part (4) of Theorem \ref{T:H(S)} above shows that 
    the colligation matrix $\bU$ is unitarily equivalent to the model 
    colligation matrix $\bU_{S}$ (see \cite{bb} for precise details).
    As a consequence 
    of part (1) of Theorem \ref{T:H(S)}, the conclusion that $S$ has maximal 
    factorable minorant equal to $0$ follows.  This completes the proof of 
    Theorem \ref{T:H(S)5}.
\end{proof}

\begin{remark}  \label{R:cni} {\em  Let $S \in \cS(\cU, \cY)$ be any 
    Schur-class function (possibly with nonzero maximal factorable 
    minorant) and let $A$ be the model operator $A = R_{0} = 
    \bS_{\cY}^{*}|_{\cH(S)}$ on $\cH(S)$.  Theorem \ref{T:cni} tells us 
    that $A$ is completely non-isometric.  Therefore one can apply the 
    construction of Theorem \ref{T:H(S)5} to arrive at a Schur-class 
    function $S_{0} \in \cS(\cU_{0}, \widetilde \cY_{0})$ 
    so that $A$ is unitarily equivalent to $R_{0} = 
    \bS_{\cY_{0}}^{*}|_{\cH(S_{0})}$ on 
    $\cH(S_{0})$ where $S_{0}$ has  the additional 
    property that $I - S_{0}^{*} S_{0}$ has zero maximal 
    factorable minorant.  Since the original $S$ for which $I - 
    S^{*} S$ does  not have zero 
    maximal factorable minorant, it cannot be the case that $S$ and 
    $S_{0}$ are the same.  A natural question is:  how does 
    $S$ determine $S_{0}$? As can be seen from the results of 
    \cite{BK} (details left to the reader),
    the answer is:  one choice of $S_{0}$ is 
    $S_{0} = \left[ \begin{smallmatrix} S \\ \Psi 
\end{smallmatrix} \right]$ where $\Psi$ is the outer factor for the 
maximal factorable minorant of $I - S^{*} S$.
} \end{remark}

\begin{remark}  \label{R:extremepoints} {\em The zero 
    maximal-factorable-minorant property is  also closely intertwined 
    with the characterization of the extreme points of the closed unit 
    ball of $H^{\infty}(\cU, \cY)$, i.e., of the Schur class 
    $\cS(\cU, \cY)$.  In the scalar case ($\cU = \cY = {\mathbb C}$), 
    it is well known that a given function $s$ is an extreme point of 
    $\cS({\mathbb C})$ exactly when $\log ( 1 - |s(\zeta)|^{2})$ is 
   log-integrable on ${\mathbb T}$ which in turn is equivalent to $1 
   - |s(\zeta)|^{2}$ having a factorization $a(\zeta)^{*} a(\zeta)$ 
   with $a$ a nonzero element of  $\cS({\mathbb C})$ which in turn 
   (in the scalar case) is equivalent to $0$ not being the maximal 
   factorable minorant for $1 - |s(\zeta)|^{2}$.  In \cite{BK} it was 
   observed for the case of $S \in \cS(\cU, \cY)$ that $I - 
   S(\zeta)^{*} S(\zeta)$ having a zero maximal factorable 
   minorant is a sufficient condition for $S$ to be an extreme point 
   of $\cS(\cU, \cY)$, and by symmetry it is also sufficient that $I 
   - S(\zeta) S(\zeta)^{*}$ have a zero maximal factorable minorant.  It 
   was then conjectured there (with attribution to de Branges) that a 
   necessary and sufficient condition for $S$ to be extreme in 
   $\cS(\cU, \cY)$ is that at least one of $I - S(\zeta)^{*} 
   S(\zeta)$ and $I - S(\zeta) S(\zeta)^{*}$ have a zero maximal 
   factorable minorant.  This conjecture was independently confirmed 
   around the same time by Treil (see \cite{Tr89}).
    } \end{remark}

In Section \ref{S:pullback}, the overlapping space $\cL_{T}$ was defined 
by formulas \eqref{cLT1}, \eqref{cLT2} for the case of a general contraction operator $T 
    \in \cL(\cH_{1}, \cH_{2})$.  Although not discussed in Section 
    \ref{S:Toeplitz},  of course this notion applies  to the 
    situation where $T = T_{S} \in \cL(H^{2}(\cU), H^{2}(\cY))$ is 
    the analytic Toeplitz operator arising from a Schur-class 
    function $S \in \cS(\cU, \cY)$.  The following result identifies 
    the associated overlapping space $\cL_{T_{S}}$ as a reproducing 
    kernel Hilbert space.
    
    \begin{theorem}   \label{T:overlapping-RKHS}  For $S \in \cS(\cU, 
	\cY)$, the overlapping space $\cL_{T_{S}}$ defined by 
	\eqref{cLT1}, \eqref{cLT2} is a reproducing kernel Hilbert 
	spaces $\cH(K_{\cL(S)})$ with reproducing kernel $K_{\cL(S)}$ 
	given by
\begin{equation}   \label{RKoverlap}
    K_{\cL(S)}(z,w) = \frac{1}{2} \frac{ \varphi(z) + 
    \varphi(w)^{*}}{1 - z \overline{w}}
\end{equation}
where $\varphi(z)$ is given by
\begin{equation}   \label{varphi}
  \varphi(z) = \int_{{\mathbb T}} \frac{\zeta + z}{\zeta - z} (I - 
  S(\zeta)^{*}S(\zeta) ) \frac{|{\tt d} \zeta|}{2 \pi}.
\end{equation}
\end{theorem}

\begin{proof}  Note first that the formula \eqref{varphi} for 
    $\varphi$ together with straightforward algebra enables one to 
    show that
\begin{equation}   \label{form1}
    \frac{1}{2} \frac{ \varphi(z) + \varphi(w)^{*}}{1 - z 
    \overline{w}} = \int_{{\mathbb T}} \frac{1}{(1 - z 
    \overline{\zeta}) (1 - \overline{w} \zeta)} (I  - S(\zeta)^{*} 
    S(\zeta)) \frac{| {\tt d} \zeta|}{2 \pi}.
\end{equation}
On the other hand,  Proposition \ref{P:cLT} identifies $\cL_{T_{S}}$ as 
    the lifted-norm space $\cH^{\ell}_{I - T_{S}^{*} T_{S}}$. For 
    $f = (I - T_{S}^{*} T_{S}) g \in \cL_{T_{S}}$ and $u \in \cU$ 
    and $w \in {\mathbb D}$, 
\begin{align*}
\langle f(w), u \rangle_{\cU} & = \langle (I - T_{S})^{*} T_{S}) g, 
k_{\rm Sz}(\cdot, w) u \rangle_{H^{2}(\cU)} \\
& = \langle f, (I - T_{S}^{*} T_{S}) k_{\rm Sz}(\cdot, w) u 
\rangle_{H^{2}(\cU)}.
\end{align*}
As the set of all $f$ of the form $f = (I - T_{S}^{*} T_{S}) g$ is dense 
in $\cH^{\ell}_{I - T_{S}^{*} T_{S}}$, it follows that the kernel 
function for $\cL_{T_{S}}$ is equal to $K_{\cL(S)}( \cdot, w) u : = (I 
- T_{S}^{*} T_{S}) k_{\rm Sz}(\cdot, w) u$.  This object can be 
computed as follows: for $u, u' \in \cU$, 
\begin{align*}
    \langle  K_{\cL(S)}(z, w) u, u' \rangle_{\cU} & = \left\langle (I 
    - T^{*}_{S} T_{S}) k_{\rm Sz}(\cdot, w) u, k_{\rm Sz}(\cdot, z) 
    u' \right\rangle_{H^{2}(\cU)} \\
  & = \int_{{\mathbb T}} \frac{1}{1 - \zeta \overline{w}} \cdot 
    \frac{1}{1 - \overline{\zeta z}}  \cdot  \left\langle (I - S(\zeta)^{*} 
 S(\zeta)) u, u' \right\rangle_{\cU} \frac{ |{\tt d}\zeta|}{ 2 \pi}.
\end{align*}
Comparison of this expression with \eqref{form1} and using that 
$|\zeta| = 1$ for $\zeta \in {\mathbb T}$ leads to the conclusion 
that indeed $K_{\cL(S)}$ is given by \eqref{RKoverlap} as claimed.
\end{proof}

\section{The de Branges-Rovnyak space $\cD(S)$}     \label{S:D(S)}

This section gives a brief discussion of the de Branges-Rovnyak 
space $\cD(S)$.  The first point of discussion is the RKHS point of 
view; there follows an elaboration of the additional insights coming from viewing $\cD(S)$ 
(or 
rather a certain minor modification) as a pullback space.

\subsection{$\cD(S)$ as a reproducing kernel Hilbert space}  
\label{S:D(S)RKHS}

Given a Schur-class function $S \in \cD(S)$, one can define a kernel 
$\widehat K_{S}$ as in \eqref{DS-ker}.  Unlike the case for $K_{S}$ 
(see the discussion around \eqref{contrac}), 
it is not immediately obvious why $\widehat K_{S}$ is a positive kernel.
One way to see this is as follows.  By the earlier discussion around 
\eqref{contrac}, one can use the connection with 
Toeplitz operators acting on $H^{2}$ spaces to see that $K_{S}$ is a 
positive kernel.  It is known that a Schur-class function $S$ has a {\em unitary} 
realization 
$$
  S(z) = D + z C (I - zA)^{-1} B \quad\text{with}\quad \bU = \begin{bmatrix} 
  A & B \\ C & D \end{bmatrix} \text{ unitary};
$$
one can cite the result of Sz.-Nagy-Foias \cite[Theorem VI.3.1]{NF} 
or adapt the lurking isometry argument to get a unitary realization 
\cite[Theorem 2.1]{Ball-Winnipeg}.  Once this is done, it is a direct 
calculation using the relations associated with $\bU$ being unitary 
to get the Kolmogorov decomposition for $\widehat K(z,w)$
\begin{equation}  \label{D(S)Kol}
  \widehat K_S(z,w) = \begin{bmatrix} C (I - zA)^{-1} \\ B^{*}(I - 
  zA^{*})^{-1} \end{bmatrix} 
  \begin{bmatrix} (I - \overline{w} A^{*})^{-1} C^{*} & (I - 
      \overline{w} A)^{-1} B \end{bmatrix}.
\end{equation}

{\em Proof of parts (2) \& (3), (4) in Theorem \ref{T:D(S)} based on 
part (1):}
Once it is known that $\widehat K_{S}$ is a positive kernel (in this case 
due to the Kolmogorov decomposition \eqref{D(S)Kol}), then it is also 
known that 
$\widehat K_{S}$ generates a RKHS $\cD(S) = \cH(\widehat K_{S})$ and hence has 
its canonical Kolmogorov decomposition
\begin{equation}   \label{canonicalD(S)Kol}
  \widehat K_{S}(z,w): = \begin{bmatrix} \frac{I - S(z) S(w)^{*}}{1 - 
  z \overline{w}}  & \frac{ S(z) - S(\overline{w})}{z - \overline{w}}  
  \\  \frac{\widetilde S(z) - \widetilde S(\overline{w})}{z - 
  \overline{w}} & \frac{I - \widetilde S(z) \widetilde 
  S(\overline{w})}{1 - z \overline{w}} \end{bmatrix} =
  \begin{bmatrix} \be_{1}(z) \\ \be_{2}(z) \end{bmatrix}
      \begin{bmatrix}  \be_{1}(w)^{*} & \be_{2}(w)^{*} \end{bmatrix}
\end{equation}
where 
$$
\be_{1}(z) \colon \begin{bmatrix} f \\ g \end{bmatrix} \mapsto f(z), 
\quad \be_{2}(z) \colon \begin{bmatrix} f \\ g \end{bmatrix} \mapsto 
g(z) 
$$
are the point-evaluation maps in the first and second components 
respectively.  A rearrangement of the matricial identity 
\eqref{canonicalD(S)Kol} leads to the system of equations
\begin{align*}
    z \overline{w} \be_{1}(z) \be_{1}(w)^{*} + I_{\cY} &
    = \be_{1}(z) \be_{1}(w)^{*} + S(z) S(w)^{*}, \\
    \overline{w} \be_{2}(z) \be_{1}(w)^{*} + \widetilde S(z) & = z \be_{2}(z) 
    \be_{1}(w)^{*} + S(w)^{*}, \\
  z \be_{1}(z) \be_{2}(w)^{*} + S(\overline{w}) & = \overline{w} 
  \be_{1}(z) \be_{2}(w)^{*}  + S(z), \\ 
 \be_{2}(z) \be_{2}(w)^{*} + S(\overline{z})^{*} S(\overline{w}) & = 
 z \overline{w} \be_{2}(z) \be_{2}(w)^{*} + I_{\cU}
\end{align*}
which is equivalent to the collection of inner-product identities
\begin{align*}
   & \langle \overline{w} \be_{1}(w)^{*} y, \overline{z} 
    \be_{1}(z)^{*} y' \rangle_{\cH(\widehat K_S)} + \langle y, y' 
    \rangle_{\cY} \\
    & \quad \quad =
\langle \be_{1}(w)^{*} y, \be_{1}(z)^{*} y' \rangle_{\cH(\widehat K_S)} +
\langle S(w)^{*} y, S(z)^{*} y' \rangle_{\cU}, \\
& \langle \overline{w} \be_{1}(w)^{*} y, \be_{2}(z)^{*} u' 
\rangle_{\cH(\widehat K_S)} + \langle \widetilde S(z) y, u' 
\rangle_{\cU} \\
& \qquad \qquad = \langle \be_{1}(w)^{*} y, \overline{z} \be_{2}(z)^{*} 
u' \rangle_{\cH(\widehat K_S)} + \langle S(w)^{*} y, u' \rangle_{\cU}, \\
& \langle \be_{2}(w)^{*} u, \overline{z} \be_{1}(z)^{*} y' 
\rangle_{\cH(\widehat K_S)}  + \langle S(\overline{w}) u, y' 
\rangle_{\cY}\\
& \quad \quad =\langle  \overline{w} \be_{2}(w)^{*} u, \be_{1}(z)^{*} y' 
\rangle_{\cH(\widehat K_S)} + \langle  u,S(z)^{*} y' \rangle_{\cU}, \\
& \langle \be_{2}(w)^{*} u, \be_{2}(z)^{*} u' \rangle_{\cH(\widehat 
K_S)} + \langle S(\overline{w}) u, S(\overline{z}) u' \rangle_{\cY} \\
& \quad \quad = \langle \overline{w} \be_{2}(w)^{*} u, \overline{z} 
\be_{2}(z)^{*} u' \rangle_{\cH(\widehat K_S)} + \langle u, u' 
\rangle_{\cU}.
\end{align*}
These in turn can be rearranged in aggregate form
\begin{align*}
& \left\langle \begin{bmatrix} \overline{w} \be_{1}(w)^{*} y + 
\be_{2}(w)^{*} u \\ y + S(\overline{w}) u \end{bmatrix},
\begin{bmatrix} \overline{z} \be_{1}(z)^{*} y' + \be_{2}(z)^{*} u' \\ 
   y' + S(\overline{z}) u' \end{bmatrix} 
   \right\rangle_{\cH(\widehat K_S) \oplus \cY}  \\
    & \quad =
    \left \langle \begin{bmatrix} \be_{1}(w)^{*} y + \overline{w} 
    \be_{2}(w)^{*} u \\ S(w)^{*} y + u \end{bmatrix},
 \begin{bmatrix} \be_{1}(w)^{*} y' + \overline{z} \be_{2}(z)^{*} u' \\
     S(z)^{*} y' + u' \end{bmatrix} \right\rangle_{\cH(\widehat K_S) 
     \oplus \cU}.
\end{align*}
Since these inner products match up, the map $V$ defined by
\begin{equation}  \label{defV'}
  V \colon \begin{bmatrix} \overline{w} \be_{1}(w)^{*} y + 
  \be_{2}(w)^{*} u \\ y + S(\overline{w}) u \end{bmatrix} \mapsto
  \begin{bmatrix} \be_{1}(w)^{*} y + \overline{w} \be_{2}(w)^{*} u \\
      S(w)^{*} y + u \end{bmatrix}
\end{equation}
extends by linearity and continuity to an isometry (also denoted 
by $V$) from 
$$
 \cD = \overline{\operatorname{span}}
 \left\{ \begin{bmatrix} \overline{w} \be_{1}(w)^{*} y + 
\be_{2}(w)^{*} u \\ y + S(\overline{w}) u \end{bmatrix} \colon
 u \in \cU, y \in \cY,  w \in {\mathbb D} \right\} \subset 
 \begin{bmatrix} \cH(\widehat K_S) \\ \cY \end{bmatrix}
$$
onto 
$$
  \cR: =\overline{\operatorname{span}}
  \left\{ \begin{bmatrix} \be_{1}(w)^{*} y + \overline{w} \be_{2}(w)^{*} u \\
      S(w)^{*} y + u \end{bmatrix} \colon u \in \cU, y \in \cY, w \in 
      {\mathbb D} \right\} \subset \begin{bmatrix} \cH(\widehat K_S) \\ 
      \cU \end{bmatrix}.
$$
By taking $y=0$ and $w = 0$ in the expression for the generic element 
of $\cD$, one can see that $\cD \supset \left[ \begin{smallmatrix} \{0\} 
\\ \cY \end{smallmatrix} \right]$.  As $u \in \cU$ and $y \in \cY$ 
are independent of each other, it follows that the projection down to the 
first component contains all the kernel functions 
$$
\widehat K(\cdot, 
w) \left[\begin{smallmatrix} y \\ u \end{smallmatrix} \right] = 
\be_{1}(w)^{*} y + \be_{2}(w)^{*} u,
$$
and hence $\cD$ in fact is all 
of $\cH(\widehat K) \oplus \cY$.  Similarly one can see that $\cR$ is 
all of $\cH(\widehat K) \oplus \cU$, and hence $V$ in fact is a 
unitary transformation from $\cH(\widehat K_S) \oplus \cY$ onto 
$\cH(\widehat K_S) \oplus \cU$. From \eqref{defV'} it follows that
\begin{equation}   \label{bUact1}
    \begin{bmatrix} A^{*} & C^{*} \\ B^{*} & D^{*} \end{bmatrix}
  \begin{bmatrix} \overline{w} \be_{1}(w)^{*} y + 
  \be_{2}(w)^{*} u \\ y + S(\overline{w}) u \end{bmatrix} =
  \begin{bmatrix} \be_{1}(w)^{*} y + \overline{w} \be_{2}(w)^{*} u \\
      S(w)^{*} y + u \end{bmatrix}.
\end{equation}
As $V$ is actually unitary, it also follows that
\begin{equation}   \label{bUact2}
    \begin{bmatrix} A & B \\ C & D \end{bmatrix}
	 \begin{bmatrix} \be_{1}(w)^{*} y + \overline{w} \be_{2}(w)^{*} u \\
      S(w)^{*} y + u \end{bmatrix} =
 \begin{bmatrix} \overline{w} \be_{1}(w)^{*} y + 
  \be_{2}(w)^{*} u \\ y + S(\overline{w}) u \end{bmatrix}.
\end{equation}
Upon setting $u = 0$ in \eqref{bUact1}, one arrives at
\begin{equation}   \label{bUact1'}
    \begin{bmatrix} A^{*} & C^{*} \\ B^{*} & D^{*} \end{bmatrix} 
	\begin{bmatrix} \overline{w} \be_{1}(w)^{*} y \\ y 
	\end{bmatrix} = \begin{bmatrix} \be_{1}(w)^{*} y \\ S(w)^{*} 
	y \end{bmatrix}.
\end{equation}
The next step is to proceed as was done in the proof of (1) $\Rightarrow$ (4) in 
Theorem \ref{T:H(S)} to get
$$
  \be_{1}(w)^{*} y = (I - \overline{w} A^{*})^{-1} C^{*} y
$$
from the first row of \eqref{bUact1} and then use this in the second 
row to get
$$
 \overline{w} B^{*} (I - \overline{w}A^{*})^{-1} C^{*} y + D^{*} y = 
 S(w)^{*} y.
$$
Then taking adjoints and setting $z = \overline{w}$ leads to the 
unitary realization for $S$:
\begin{equation} \label{realize}
    S(z) = D + z C (I - zA)^{-1} B.
\end{equation}
Alternatively, one may set $y=0$ in \eqref{bUact2} to get
\begin{equation}  \label{bUact2'}
    \begin{bmatrix} A & B \\ C & D \end{bmatrix} \begin{bmatrix} 
	\overline{w} \be_{2}(w)^{*} u \\ u \end{bmatrix} = 
	\begin{bmatrix} \be_{2}(w)^{*} u \\ S(\overline{w}) u 
	\end{bmatrix}.
\end{equation}
The first row can be solved for $\be_{2}(w)^{*}u$ to get
$$
\be_{2}(w)^{*} u = (I - \overline{w} A)^{-1} B u.
$$
From the second row it then follows that
$$
 \overline{w} C (I - \overline{w} A)^{-1} B u + D u = S(\overline{w}) 
 u.
 $$
 Letting $z = \overline{w} \in {\mathbb D}$ then again leads to 
 the realization \eqref{realize}.  As $V$ (and $V^{*}$) is unitary, 
 either way leads to a unitary realization \eqref{realunitary} for $S$.
 It remains to identify $V^{*} = \left[ \begin{smallmatrix} A & B \\ 
 C & D \end{smallmatrix} \right]$ with the model colligation 
 $\widehat \bU_{S} = \left[ \begin{smallmatrix} \widehat A_{S} & 
 \widehat B_{S} \\ \widehat C_{S} & \widehat D_{S} \end{smallmatrix} 
 \right]$ given by \eqref{hatbU}.
 
\smallskip

 From \eqref{bUact1'} with $w = 0$ it is seen that $C^{*} y = 
 \be_{1}(0)^{*}y$. A simple adjoint computation now gives
 $$
   C \colon \begin{bmatrix} f \\ g \end{bmatrix} \to f(0),
 $$
 so $C = \widehat C_{S}$.  With the action of $C^{*}$ identified, one 
 can use the first row of \eqref{bUact1'} to solve for $A^{*} 
 \be_{1}(w)^{*} y$:
 $$
  A^{*} \be_{1}(w)^{*} y = \frac{ \be_{1}(w)^{*} - 
  \be_{1}(0)^{*}}{\overline{w}} y.
 $$
 Then a simple adjoint computation
 $$
 \left\langle A \begin{bmatrix} f \\ g \end{bmatrix}, \be_{1}(w)^{*} 
 y \right\rangle_{\cH(\widehat K)} = \left\langle \begin{bmatrix} f 
 \\ g \end{bmatrix}, \frac{\be_{1}(w)^{*} - \be_{1}(0)^{*}}{\overline{w}} 
 y \right\rangle_{\cH(\widehat K)} = \left\langle \frac{f(w) - f(0)}{w}, y \right 
\rangle_{\cY}
 $$
 reveals what the action of $A$ is in the first component. From 
 \eqref{bUact1} with $y=0$ gives
 $$
 A^{*} \be_{2}(w)^{*} u + C^{*} S(\overline{w}) u = \overline{w} 
 \be_{2}(w)^{*} y
 $$
 or
 $$ A^{*} e_{2}(w)^{*} u  = \overline{w} \be_{2}(w)^{*} u - 
 \be_{1}(0)^{*} S(\overline{w}) u.
 $$
 Thus
 \begin{align*}
     \left\langle A \begin{bmatrix} f \\ g \end{bmatrix}, 
     \be_{2}(w)^{*} u \right\rangle_{\cH(\widehat K)} & =
     \left\langle \begin{bmatrix} f \\ g \end{bmatrix}, A^{*} 
     \be_{2}(w)^{*} u \right\rangle_{\cH(\widehat K)} \\
     & = \left\langle \begin{bmatrix} f \\ g \end{bmatrix}, 
     \overline{w} \be_{2}(w)^{*} u - \be_{1}(0)^{*} S(\overline{w}) u 
     \right\rangle_{\cH(\widehat K)} \\
     & = \langle w g(w) - S(\overline{w})^{*} f(0), u \rangle_{\cU}.
 \end{align*}
 This completes the confirmation that $A = \widehat R_{0}$.

\smallskip
 
 From \eqref{bUact2} with $y=0$ and $w=0$ one sees that 
 $$
 Bu = e_{2}(0)^{*} u = \begin{bmatrix} ([S(z) - S(0)]/z) u \\ 
 K_{\widetilde S}(z,w) u \end{bmatrix} = \widehat B_{S} u.
 $$
  Finally, setting $w=0$ in \eqref{bUact2'} gives $D=S(0)$.  
  This completes the verification that $V^{*} = \widehat R_{0}$.  The verification 
 that $V = (\widehat R_{0})^{*}$ is given by \eqref{hatRS*} is symmetric 
 (interchange the roles of the components $\left[ \begin{smallmatrix} 
 f \\ g \end{smallmatrix} \right]$ in an element of $\cD(S)$).  
 The norm identities \eqref{normid's} are simple consequences of $V$ being 
 unitary. This completes the  verification of (1) $\Rightarrow$ 
 (2) \& (3), (4) in Theorem \ref{T:D(S)}.
\qed

\smallskip 

The following verification of (1) $\Rightarrow$ (5) in Theorem 
\ref{T:D(S)} very much parallels the proof in Section \ref{S:RKHS} 
for the corresponding result in Theorem \ref{T:H(S)}.

\begin{proof}[Proof of (1) $\Rightarrow$ (5) in Theorem \ref{T:D(S)}]
    Let $A$ be a completely nonunitary contraction operator on the 
    Hilbert space $\cX$. Let $\bU = \left[ \begin{smallmatrix} A 
    & B \\ C & D \end{smallmatrix} \right]$ be a {\em Julia operator} 
    for $A$ (see \cite{DR}), also known as a Halmos dilation of $A$ (see 
    \cite{Halmos})); by this is meant that $B,C,D$ are constructed so 
    that $\bU$ is unitary with $B$ injective and with $C$ having 
    dense range.  The simplest choice of $B,C,D$ is to take $\cU = 
    \cD_{A^{*}} = \overline{\operatorname{Ran}} D_{A^{*}}$ where 
    $D_{A^{*}} = (I - A A^{*})^{\frac{1}{2}}$, $\cY = \cD_{A} = 
    \overline{\operatorname{Ran}} D_{A}$ where $D_{A} = (I - 
    A^{*}A)^{\frac{1}{2}}$ and set
    $$
    \bU =\left.  \begin{bmatrix} A & D_{A^{*}} \\ D_{A} & -A^{*} 
\end{bmatrix}\right|_{\cX \oplus \cD_{A^{*}}} \colon \begin{bmatrix} 
\cX \\ \cD_{A^{*}} \end{bmatrix} \to \begin{bmatrix} \cX \\ \cD_{A} 
\end{bmatrix}.
$$
The next step is to define a map $\Xi \colon \cX \to \left[ \begin{smallmatrix} 
H^{2}(\cY) \\ H^{2}(\cU) \end{smallmatrix} \right]$ by
$$
  \Xi \colon x \mapsto \begin{bmatrix} \cO_{C,A} x \\ 
  \cO_{B^{*}, A^{*}}x \end{bmatrix} = \begin{bmatrix} C (I - 
  zA)^{-1} x \\ B^{*} (I - zA^{*})^{-1} x \end{bmatrix}.
$$
The assumption that $A$ is completely nonunitary guarantees that 
$\Xi$ is injective.  Let $\cD_{0}$ be the range of the map $\Xi$ 
with the pullback norm:
$$
  \| \Xi x \|_{\cD_{0}} = \| x \|_{\cX}.
$$
Just as in the proof done in Section \ref{S:RKHS} for (1) 
$\Rightarrow$ (5) in Theorem \ref{T:H(S)}, one can verify that 
$\cD_{0}$ is a reproducing kernel Hilbert space with reproducing 
kernel $K_{C,A,B}$ given by
$$
   K_{C,A,B}(z,w) = \begin{bmatrix} C (I - zA)^{-1} \\ B^{*} (I - z 
   A^{*})^{-1} \end{bmatrix} \begin{bmatrix} I - \overline{w} 
   A^{*})^{-1} C^{*} & (I - \overline{w} A)^{-1} B \end{bmatrix}.
$$
Set $S(z) = D + z C (I - zA)^{-1} B$; then $S$ is in the Schur 
class $\cS(\cU, \cY)$ and, as was already observed above (see 
\eqref{D(S)Kol}), the fact that $\bU = \left[ \begin{smallmatrix} A & 
B \\ C & D \end{smallmatrix} \right]$ is unitary implies that 
$K_{C,A,B}(z,w) = \widehat K_{S}(z,w)$.  It is now a straightforward 
(if perhaps tedious) exercise to verify that $\widehat R_{0} \Xi = 
\Xi A$.  As $\Xi \colon \cX \to \cD_{0} = \cD(S)$ is unitary, it follows 
that $A$ is unitarily equivalent to the operator $\widehat 
R_{0}$ on $\cD(S)$.
\end{proof}

\begin{remark}  \label{R:NFcharfunc}  {\em
    Those familiar with the Sz.-Nagy--Foias model theory \cite{NF} 
    will  notice that the function $S$ produced in the preceding 
    proof is just the Sz.-Nagy--Foias characteristic function for the 
    operator $A^{*}$.
} \end{remark}

\subsection{$\cD(S)$ as a pullback space}  \label{D(S)pullback}

By definition,  the two-component de Branges-Rovnyak space $\cD(S)$ 
sits inside the direct sum Hardy space $H^{2}(\cY) \oplus 
H^{2}(\cU)$.  Nikolskii and Vasyunin (see \cite{NV1, NV2}) have 
introduced an adjusted version which sits inside $H^{2}(\cY)  \oplus 
\left(L^{2}(\cU) \ominus H^{2}(\cU)\right)$ which is a more natural object to look 
at in the context of Sz.-Nagy dilation theory and Lax-Phillips 
scattering as will be now described.

\smallskip

The adjustment is simple enough:  note that the map $j \colon 
f(\zeta) \mapsto \overline{\zeta} f(\overline{\zeta})$ is an involution on 
$L^{2}({\mathbb T}, \cU)$ which transforms $H^{2}(\cU)$ to 
$H^{2}(\cU)^{\perp}: = L^{2}(\cU) \ominus H^{2}(\cU)$ and vice versa. 
Use the notation $\widetilde D(S)$ for the flip of the space $\cD(S)$ defined by
$$
 \widetilde \cD(S) = \left\{ \begin{bmatrix} f \\ g \end{bmatrix} 
 \colon \begin{bmatrix} f \\ j ( g ) \end{bmatrix} \in \cD(S) \right\}.
$$
Then it is shown in \cite{NV1} that $\widetilde \cD(S)$ can be 
identified with the pullback space
$$ \widetilde \cD(S) = \left. \operatorname{Ran} \begin{bmatrix} 
I_{H^{2}(\cY)} & L_{S} \\ L_{S}^{*} & I_{H^{2}(\cU)^{\perp}} 
\end{bmatrix} \right|_{H^{2}(\cY) \oplus H^{2}(\cU)^{\perp}}
$$
with the pullback norm
$$
  \left \| \begin{bmatrix} I & S \\ S^{*} & I \end{bmatrix} 
  \begin{bmatrix} f \\ g \end{bmatrix}  \right\|_{\widetilde \cD(S)} 
      = \left\| \bQ  \begin{bmatrix} f \\ g \\ \end{bmatrix} 
      \right\|_{H^{2}(\cY) \oplus H^{2}(\cU)^{\perp}}, \text{ where }
      \bQ = P_{\left(\operatorname{Ker} \left[ \begin{smallmatrix} I & 
      L_{S} \\ L_{S}^{*} & I \end{smallmatrix} \right] 
      \right)^{\perp}},
$$
or, in the notation of Section \ref{S:pullback},
\begin{equation}   \label{tD(S)}
\widetilde \cD(S) = \cH^{p}_{\left. \left[ \begin{smallmatrix} I & L_{S} \\ 
L_{S}^{*} &  I \end{smallmatrix} \right]\right|_{H^{2}(\cY) \oplus 
H^{2}(\cU)^{\perp}}}.
\end{equation}

The utility of this formulation is that one can see better the 
unitary dilation space and the unitary dilation of the model operator 
$\widehat R_{0}$ as well as the associated Lax-Phillips scattering.
Recall that, for $T$ a contraction operator on $\cH$, by a theorem of 
Sz.-Nagy (see \cite{NF, Douglas74}), $T$ has a unitary dilation $U$, 
i.e., there is a unitary operator $U$ on a Hilbert space $\cK 
\supset \cH$ such that
$$
 T^{n} = P_{\cH} U^{n}|_{\cH} \quad\text{for}\quad n=0,1,2, \dots.
$$
By Sarason's lemma \cite{Sarason65}, the space $\cK$ has a  decomposition 
\begin{equation}   \label{scatdecom}
 \cK = \cG_{*} \oplus \cH \oplus \cG\quad\mbox{where}\quad 
U^{*} \cG_{*} \subset \cG_{*},\quad U \cG \subset \cG.
\end{equation}
When the unitary dilation $U$ is minimal, i.e., when $\cK = 
\overline{\operatorname{span}} \{ U^{n} \cH \colon n \in {\mathbb 
Z}\}$,   then necessarily $U^{*}|_{\cG_{*}}$ and $U|_{\cG}$ are pure 
isometries, and hence
$$
 \cG_{*} = \bigoplus_{n \ge 0} U^{*n}(\cG_{*} \ominus U^{*} \cG_{*}), 
 \quad  \cG = \bigoplus_{n \ge 0} U^{n}(\cG \ominus U \cG).
$$
When the completely nonunitary contraction operator $T$ is modeled 
as the model operator
$$
\left( \widetilde{\widehat R_{0}} \right)^{*} \colon \begin{bmatrix} f(\zeta) 
 \\ g(\zeta) \end{bmatrix} \mapsto \begin{bmatrix} \zeta f(\zeta) - 
S(\zeta) [g]_{-1} \\ \zeta g(\zeta) - [g]_{-1} \end{bmatrix}
$$
on $\widetilde \cD(S)$ (where now the functions in the model 
space are written as functions of the variable $\zeta$ on the circle ${\mathbb 
T}$), then it can be shown that the unitary 
dilation  for $T = \left( \widetilde{\widehat R_{0}} \right)^{*}$ can 
be modeled as the operator 
$$
 M_{\zeta} \colon \begin{bmatrix} f(\zeta) \\ g(\zeta) \end{bmatrix} 
 \mapsto \begin{bmatrix} \zeta f(\zeta) \\ \zeta g(\zeta) 
\end{bmatrix}
$$
on the pullback space
\begin{equation}   \label{tK(S)}
  \widetilde \cK(S): = \cH^{p}_{ \left[ \begin{smallmatrix} I & L_{S} 
  \\ L_{S}^{*} & I \end{smallmatrix} \right]}
\end{equation}
where the operator $\left[ \begin{smallmatrix} I & 
L_{S} \\ L_{S}^{*} & I \end{smallmatrix} \right] $ is now viewed as acting of 
$L^{2}(\cY) \oplus L^{2}(\cU)$.  The decomposition \eqref{scatdecom}
is valid with $\cK = \widetilde \cK(S)$ as in \eqref{tK(S)}, $\cH = 
\widetilde \cD(S)$ as in \eqref{tD(S)}, and with the incoming space 
$\cG_{*} = \widetilde \cG_{*}(S)$ and the outgoing space $\cG = 
\widetilde \cG(S)$ given respectively by
\begin{equation}   \label{in/out}
  \widetilde \cG_{*}(S) = \cH^{p}_{\left. \left[ \begin{smallmatrix} I \\ 
  L_{S}^{*} \end{smallmatrix} \right] \right| H^{2}(\cY)^{\perp}}\quad\mbox{and}\quad
  \quad\widetilde \cG(S) = \cH^{p}_{\left. \left[ \begin{smallmatrix} 
  L_{S} \\ I \end{smallmatrix} \right] \right| H^{2}(\cY)}.
\end{equation}

In the Nikolskii-Vasyunin terminology, there is a coordinate-free 
formulation of the model for a completely nonunitary contraction and 
the associated unitary dilation, and this de Branges-Rovnyak 
formulation is but one of three possible transcriptions, the others 
being the Sz.-Nagy-Foias and the Pavlov transcriptions. A thorough 
extension of all these ideas to multievolution scattering systems and 
Schur-class functions on the polydisk (rather than on the unit disk), 
where the realization theory has man more subtleties and 
complications, is carried out in \cite{BabyBear}.

\subsection{Two-component overlapping spaces: factorization and 
invariant subspaces}   The following enhanced generalization 
of Proposition \ref{T:overlapping-RKHS} is  relevant to these issues.

\begin{theorem}  \label{T:overlapping-2comp}  (See \cite{dB70, 
    Ball-Memoir, BC}).  Let $S_{1} \in 
    \cS)(\cU_{0}, \cU)$ and $S_{2} \in \cS(\cU, \cY)$ be two 
    Schur-class functions for which the product $S = S_{2} \cdot 
    S_{1} \in \cS(\cU_{0}, \cY)$ is defined.  Then the map
    $$
   Z \colon \begin{bmatrix} f_{2} \\ g_{2} \end{bmatrix} \oplus 
	\begin{bmatrix} f_{1} \\ g_{1} \end{bmatrix} \mapsto 
	    \begin{bmatrix} f_{2}(z) + S_{2}(z) f_{1}(z) \\ 
		\widetilde S_{1}(z) 
		g_{2}(z) + g_{1}(z) \end{bmatrix}
$$
is a coisometry from $\cD(S_{2}) \bigoplus \cD(S_{1})$  onto 
$\cD(S_{2} S_{1})$.  If one defines the two-component
overlapping space 
$\cE(S_{2} \cdot S_{1})$ by
$$
\cE(S_{2} \cdot S_{1}) = \left\{ \begin{bmatrix} f \\ g \end{bmatrix} 
\in H^{2}(\cU) \colon \begin{bmatrix} S_{2} f \\ - g \end{bmatrix} 
\in \cD(S_{2}) \text{ and } \begin{bmatrix} -f \\ \widetilde S_{1} g 
\end{bmatrix} \in \cD(S_{1}) \right\}
$$
with norm
$$
  \left\| \begin{bmatrix} f \\ g \end{bmatrix} \right\|^{2} = \left\| 
  \begin{bmatrix} S_{2} f \\ - g \end{bmatrix} \right\|^{2} + \left\| 
      \begin{bmatrix} - f \\ \widetilde S_{1} g \end{bmatrix} 
	  \right\|^{2},
$$
then the map 
$$  \chi \colon \begin{bmatrix} f \\ g \end{bmatrix} \mapsto 
\begin{bmatrix} S_{2} f \\ - g \end{bmatrix} \oplus  \begin{bmatrix} 
  -f \\  \widetilde S_{1} g \end{bmatrix}
$$
is a unitary embedding of $\cE(S_{2} \cdot S_{1})$ into 
$\operatorname{Ker} Z$.
Furthermore, $\cE(S_{2} \cdot S_{1})$ is RKHS
with reproducing kernel $K_{S_{2} \cdot S_{1}}$ given by
$$
K_{S_{2} \cdot S_{1}}(z,w)  = \begin{bmatrix} \frac{1}{2} (1 - z 
\overline{w})^{-1} (\varphi(z) + \varphi(w)^{*}) & \frac{1}{2} (z - 
\overline{w})^{-1} (\varphi(z) - \varphi(\overline{w}))  \\
\frac{1}{2} (z - \overline{w})^{-1} (\widetilde \varphi(z) - 
\widetilde \varphi(\overline{w}) & \frac{1}{2} (1 - z 
\overline{w})^{-1} (\widetilde \varphi(z) + \widetilde 
\varphi(w)^{*})  \end{bmatrix}
$$
where in this case $\varphi$ is given by
$$
  \varphi(z) = \int_{{\mathbb T}} \frac{ \zeta + z} {\zeta - z} 
 \,  \Omega(\zeta) \, \frac{| {\tt d} \zeta|}{2 \pi}
$$
and $\Omega(\zeta)$ is a certain positive semidefinite operator on 
${\mathbb T}$ determined by the two pointwise defect operators
$\Omega_{2}(\zeta): =  I - S_{2}(\zeta)^{*} S_{2}(\zeta) $ and 
$\Omega_{1}(\zeta) : = I - S_{1}(\zeta) S_{1}(\zeta)^{*}$; 
when $\Omega_{2}(\zeta)$ and $\Omega_{1}(\zeta)$ are both invertible, 
then $\Omega(\zeta)$ is determined from the identity
\begin{equation}   \label{dBparallelsum}
  \Omega(\zeta)^{-1} = \Omega_{2}(\zeta)^{-1} + 
  \Omega_{1}(\zeta)^{-1} - I.
\end{equation}
\end{theorem}

\begin{remark} \label{R:invsubspaces} {\em  
    It turns out that Theorem \ref{T:overlapping-2comp} 
    is connected with factorization and invariant subspaces. 
    A sketch of the explanation with getting into precise details is 
    as follows.
    As explained in Brodskii \cite{Brod78} and Ball-Cohen 
    \cite{BC}, invariant subspaces for the state operator $\widehat 
    R_{0}^{*}$ on $\cD(S)$ are determined by nontrivial {\em  
    regular} factorizations $S = S_{2} \cdot S_{1}$ of the 
    characteristic function $S$.  The factorization $S = S_{2} \cdot 
    S_{1}$ being regular corresponds to the overlapping 
    space $\cE(S_{2} \cdot S_{1})$ being trivial, or to $Z$ being 
    unitary from $\cD(S_{2}) \oplus \cD(S_{1})$ to $\cD(S)$ (see 
    \cite{Ball-Memoir} and \cite[Section 7]{BC}), or to $Z$ being 
    unitary from $\cD(S_{2}) \oplus \cD(S_{1})$ to $\cD(S)$.  Tracing 
    through the form of the operator $\widehat R_{0}^{*}$ on 
    $\cD(S)$ in the alternative decomposition $\cD(S_{2}) \oplus 
    \cD(S_{1})$ obtained by applying $Z^{*}$ to $\cD(S)$,  one can 
    see that $0 \oplus \cD(S_{1})$ is an invariant subspace of 
    $\widehat R_{0}^{*}$; by the general theory of cascade 
    decompositions of unitary colligations (see \cite{Brod78} and 
    \cite{BC}), every invariant subspace 
    for $\widehat R_{0}^{*}$ arises in this way.  The problem of finding 
    nontrivial invariant subspaces for a completely nonunitary 
    contraction operator is therefore transferred to the problem of 
    finding nontrivial regular factorizations for characteristic 
    operator functions; as these Schur-class functions in general act 
    between infinite-dimensional coefficient Hilbert spaces, this 
    problem in turn is tractable only with additional assumptions.  
    it is a curious fact, nonetheless, that even when the 
    factorization is not regular, one still gets an invariant 
    subspace, but for $\widehat R_{0}^{*} \oplus U$ on $\cD(S) \oplus 
    \cE(S_{2} \cdot S_{1})$ rather than for $\widehat R_{0}^{*}$ 
    itself; here $U$ is the unitary operator on $\cE(S_{2} \cdot 
    S_{1})$ given by
    $$
      U \colon \begin{bmatrix} f(z) \\ g(z) \end{bmatrix} \mapsto 
      \begin{bmatrix}  z f(z) - g(0) \\ [g(z) - g(0)]/z \end{bmatrix}.
$$
This result was also obtained independently in the 
setting of the Sz.-Nagy-Foias model (see \cite[Notes to Chapter VII]{NF}).
While this phenomenon appears to be disappointing from the point of 
view of searching for invariant subspaces, it is exactly the tool 
used in \cite{Ball-Memoir} to obtain the spectral invariants for the 
unitary part of a whole class of contractive integral operators 
defined on a vector-valued $L^{2}$-space on the unit circle (see also 
\cite{Kriete} for a real-line version). 

\smallskip

Finally, it turns out that  the operator $\Omega(\zeta)$ appearing 
in Theorem \ref{T:overlapping-2comp} satisfies
\begin{equation}  \label{Ranintersect}
    \operatorname{Ran} \Omega(\zeta)^{1/2} = \operatorname{Ran} 
    \Omega_{2}(\zeta)^{1/2} \cap \operatorname{Ran} 
    \Omega_{1}(\zeta)^{1/2}.
\end{equation}
Thus the operator $\Omega(\zeta)$ is related to but not quite the 
same as the parallel sum of $\Omega_{1}(\zeta)$ and 
$\Omega_{2}(\zeta)$ studied by Fillmore and Williams \cite{FW} with 
motivation from circuit theory; the parallel sum studied in \cite{FW}
also satisfies the range-intersection property \eqref{Ranintersect} 
but is determined in simple cases by the parallel-sum identity
\begin{equation}  \label{parallelsum}
    \Omega(\zeta)^{-1} = \Omega_{2}(\zeta)^{-1} + 
    \Omega_{1}(\zeta)^{-1}
\end{equation}
rather than by the Brangesian parallel-sum identity 
\eqref{dBparallelsum}.  Nevertheless, a consequence of the range 
intersection property \eqref{Ranintersect} is that the overlapping 
space $\cE(S_{2} \cdot S_{1})$ is trivial, i.e., the factorization $S 
= S_{2} \cdot S_{1}$ is regular, if and only if 
$$
\operatorname{Ran} \Omega_{2}(\zeta)^{1/2} \cap \operatorname{Ran} 
\Omega_{1}(\zeta)^{1/2} = \{0\} \text{ for a.e. } \zeta \in {\mathbb 
T}
$$
(see \cite{Ball-Memoir}).  An independent direct proof for this 
factorization-regularity criterion was given in the setting of the 
Sz.-Nagy-Foias model theory in \cite{NF74}.
} \end{remark}

\section{Generalizations and extensions}

\subsection{Canonical de Branges-Rovnyak functional-model spaces:
multivariable settings} \label{S:multivariable}

Realization  of a Schur-class function as the transfer function of a
canonical functional-model colligation having additional metric
properties (e.g., coisometric, isometric, or unitary), i.e., item (4)
in Theorems \ref{T:H(S)} an \ref{T:D(S)}, has been extended to
settings where the unit disk playing the role of the underlying
domain is replaced by a more
general domain $\cD$ in ${\mathbb C}^{d}$; see \cite{bbBall} for the case of the unit ball
${\mathbb B}^{d}$ in ${\mathbb C}^{d}$,  \cite{bbPoly} for the case
of the unit polydisk ${\mathbb D}^{d}$, \cite{bbGen} for the case of
a general domain with matrix polynomial defining function, and
\cite{bb} for an overview of all three settings.  In these
multivariable settings, the backward shift operator $R_{0}$ is
replaced by a solution of the Gleason problem; an early manifestation 
of this idea is in \cite{ADR-IEOT}.  For the case where
the origin is in the domain $\cD \subset {\mathbb C}^{d}$, the Gleason
problem (centered at $0$) can be formulated as:  given $f$ in our
space of functions $\cF$ on $\cD$, find $f_{1}, \dots, f_{d}$ also in
$\cF$ so that $f(z)  = f(0) + \sum_{k=1}^{d} z_{k}
f_{k}(z)$.    As the solution of such a Gleason problem is often not
unique, one speaks about many de Branges-Rovnyak spaces $\cH(S)$ or
$\cD(S)$ associated with a given function $f$ in the generalized
Schur-class over the domain $\cD$; as long as certain minimal
structural components are maintained, all these are called canonical
functional models going with the same $S$.  One then gets good
uniqueness results in the following sense:  any other
transfer-function realization satisfying certain
observability/controllability and weak metric properties is
unitarily equivalent to some functional-model transfer-function
realization.  

\smallskip

There has also been work on extending the functional-model aspect
(item (5) in Theorems \ref{T:H(S)} and \ref{T:D(S)}), at least in the
ball setting, where a commutative row contraction , i.e., a
commutative $d$-tuple of operators $T_{1}, \dots, T_{d}$ on a Hilbert   
space $\cH$ for which the block row $\begin{bmatrix} T_{1} & \cdots &
T_{d} \end{bmatrix} \colon \cH^{d} \to \cH$ is contractive, replaces
a single contraction operator $T$ (see \cite{BES, bbBall}).  There has
also been extensive work on noncommutative versions (models for a not
necessarily commutative operator $d$-tuple with block-row matrix $\begin{bmatrix}
T_{1} & \cdots & T_{d} \end{bmatrix}$ contractive---see
\cite{Bunce, Frazho, Popescu1,  Popescu2, Popescu3,
Cuntz-scat, BBF3}) which then leads into noncommutative function
theory.  For lack of space,  these matters are not dealt with in any 
detail here. 

\subsection{Extensions to Kre\u{\i}n space settings}
Much of the theory of de Branges-Rovnyak spaces given in Sections
\ref{S:dBRspaces} and \ref{S:D(S)} actually extends to Pontryagin and
Kre\u{\i}n-space settings, where  Hilbert spaces coming up in
various places are allowed to be Kre\u{\i}n spaces (i.e., the space
is a direct sum of a Hilbert space and an anti-Hilbert space), or at
least Pontryagin spaces (where the anti-Hilbert space is finite
dimensional).  In particular, the paper of de Branges
\cite{deB88} provides a nice extension of the theory of complementary
spaces developed in Section \ref{S:dBRspaces} above to the
Kre\u{\i}n-space setting; the book
\cite{ADRdS}, besides reviewing these matters, also develops the whole realization
theorem (item (4) in Theorems \ref{T:H(S)} and \ref{T:D(S)}) to the
Pontryagin-space setting (see also \cite{DR}).  One place   where these generalizations 
are
relevant is in the proof of the Bieberbach conjecture.  Certain
relevant inequalities  involve contraction operators on a
Pontryagin function space, involving   substitution (or composition)
contraction operators $T$ rather than multiplication contraction operators $T = T_{S}$
associated with a Schur function $S$ (see \cite{dB84, dB85}).

\smallskip

The Pontryagin-space formulation of the  Nikolskii-Vasyunin
model space $\widetilde \cD(S)$ in terms of Kre\u{\i}n-Langer
representations was given in \cite{Der1, Der2} as a necessary step
to formulate a general interpolation problem for generalized Schur functions.

\section{Concluding remarks}
The preceding sections give an overview of the basic properties of 
de Branges Rovnyak spaces along with  their
applications to related function theory and operator theory problems.
It is worth noting that the theory and applications are still 
evolving, as illustrated by the following examples.

\subsection{Still other settings}
Whenever one has a substitute for the Hardy space $H^{2}(\cY)$ and
of the Schur class $\cS(\cU, \cY)$, possibly in its role as the
multipliers from $H^{2}(\cU)$ to $H^{2}(\cY)$, one has a notion of de
Branges-Rovnyak space $\cH(S)$.  An easier first case is the case
where $S$ is inner, so that $\cH(S) = H^{2}(\cY) \ominus S \cdot
H^{2}(\cU)$ is just a Hardy-space orthogonal difference.  Just as in 
the multivariable context mentioned in Section \ref{S:multivariable} 
above, one of the issues often is to find the appropriate substitute
or analog for the difference-quotient transformation $R_{0} \colon f(z) \mapsto
[f(z) - f(0)]/z$.  These ideas have been explored at least in a preliminary way in
the following situations:
\begin{enumerate}
    \item \textbf{Quaternionic settings:}   Two distinct  flavors of this topic are slice 
hyperholomorphic functions \cite{acs}, and  Fueter-regular functions \cite{asv}.

\item \textbf{Riemann-surface settings:}  See \cite{alvi, BV-Bordeaux}.
\item \textbf{Subbergman spaces:} See \cite{zhu1, zhu2} and \cite{bblerer} for
the treatment of an interpolation problem in sub-Bergman spaces.

\item \textbf{de Branges-Rovnyak spaces over a half plane:} The paper
of Ball-Kurula-Staffans-Zwart \cite{BKSZ} extends the canonical de
Branges-Rovn\-yak func\-tional-model colligation to the right half plane
setting and thereby gets canonical-model energy-preserving and
co-energy-preserving system realizations for Schur-class functions
over the right half plane.
Fricain and Mashreghi \cite{FM-Fourier} studied the boundary behavior of derivatives of
functions in a de Branges-Rovnyak space over the upper half plane.
\end{enumerate}

\subsection{Special questions}  Researchers have used de
Branges-Rovnyak spaces as a tool to treat
various types of special questions.  Examples are:
\begin{enumerate}
    \item \textbf{Riesz bases of reproducing kernels:} Given $S\in\cS(\cU,\cY)$,
$\{z_n\}\subset \D$ and $\{y_n\}\subset
\cY$, find a criterion for $\{K_S(\cdot,z_n)u_n\}$ to a Riesz basis
for $\cH(S)$. A reference for this topic is \cite{cft}.

\item \textbf{Multiplication by finite Blaschke products:} On which spaces algebraically 
included in
$H^2$ does multiplication by a finite  Blaschke product act as an
isometry? See \cite{singh}.
\end{enumerate}

\subsection{Applications}
De Branges-Rovnyak spaces appear naturally in the context
of Schur-class interpolation theory; the prominent role 
played by these spaces in interpolation theory is discussed in detail in 
separate survey \cite{bbsur}. Besides the $H^{\infty}$-interpolation theory
de Branges-Rovnyak spaces have also appeared as a useful tool in a
number of other applications.
\begin{enumerate}
    \item \textbf{Multipliers of de Branges-Rovnyak spaces:}  
    References include \cite{crof, davmcc, lotto1, lottosar1, lottosar2,
    lottosar3, suar}.
Interpolation with operator argument is embedded into the scheme
of \cite{bol}. All solutions are characterized in terms of  positive 
kernels but there is no more detailed parametrization of the solution
set at this level of generality.  Some attempts to get realizations for contractive 
multipliers
were done in \cite{albol}.

\item \textbf{Norms of weighted composition operators:}   See \cite{jury}.

\item \textbf{Relative angular derivatives:} See  \cite{shap1,shap2}
\end{enumerate}

\end{document}